%% file: Laguerre-CBOPs.tex
\documentclass{article}
\usepackage{amssymb,color, amsmath, amsthm}
\usepackage{wrapfig}
\usepackage{framed,cancel}
\definecolor{shadecolor}{rgb}{0.95, 0.95, 0.86}
\usepackage{graphicx,verbatim}
\usepackage[colorlinks=true, pdfstartview=FitV, linkcolor=blue, citecolor=blue, urlcolor=blue]{hyperref}

\def\Ai{\mathrm {Ai}}

\def\wt{\widetilde}

\def\wh{\widehat}

\def\ds{\displaystyle}
\def\res{\mathop{\mathrm{res}}\limits_}
\def\eqref#1{(\ref{#1})}

\renewcommand{\theequation}{\arabic{section}-\arabic{equation}}
\makeatletter
\@addtoreset{equation}{section}
\makeatother
\def\tr{\mathrm {Tr}}
\def\le{\left}
\def\ri{\right}
\def\QED{{\bf Q.E.D.}\par\vskip 5pt}
\def\bc{\begin{corollary}}
\def\ec{\end{corollary}}
\def\&{&{\hskip -20pt}}
\def\m{\mathop}
\def\br{\begin{remark}\rm\small}
\def\er{\end{remark}}
\def\bt{\begin{theorem}}
\def\et{\end{theorem}}
\def\bd{\begin{definition}}
\def\ed{\end{definition}}
\def\bp{\begin{proposition}\rm}
\def\bl{\begin{lemma}\em}
\def\el{\end{lemma}}
\def\ep{\end{proposition}}
\def\bea{\begin{eqnarray}}
\def\eea{\end{eqnarray}}
\def \pa{\partial}
\def\C{{\mathbb C}}
\def\R{{\mathbb R}}
\def\N{{\mathbb N}}

\newtheorem{theorem}{Theorem}[section]
\newtheorem{conjecture}{Conjecture}[section]

\newtheorem{coroll}{Corollary}[section]

\newtheorem{lemma}{Lemma}[section]
\newtheorem{remark}{Remark}[section]
\newtheorem{remarks}[remark]{Remarks}
\newtheorem{proposition}{Proposition}[section] 
\newtheorem{definition}{Definition}[section]

\def\br{\begin{remark}}
\def\er{\end{remark}}
\def\bt{\begin{theorem}}
\def\et{\end{theorem}}
\def\bc{\begin{coroll}}
\def\ec{\end{coroll}}
\def\brs{\begin{remarks} \rm\
\begin{enumerate}}
\def\ers{\end{enumerate}\end{remarks}}
\def\bl{\begin{lemma}}
\def\el{\end{lemma}}
\def\bd{\begin{definition}}
\def\ed{\end{definition}}
\def\bp{\begin{proposition}}
\def\ep{\end{proposition}}
\def\be{\begin{equation}}
\def\ee{\end{equation}}
\def\d{{\rm d}}
\def\bea{\begin{eqnarray}}
\def\eea{\end{eqnarray}}
\def \pa{\partial}
\def\iint{\int\!\!\!\!\int}
\date{}
\begin{document}
\baselineskip 16pt plus 1pt minus 1pt
\vspace{0.2cm}
\begin{center}
\begin{Large}
\textbf{Cauchy-Laguerre two-matrix model  and the Meijer-G random point field}
\end{Large}\\
\bigskip
\begin{large} {M.
Bertola$^{\dagger}$,  M. Gekhtman$^{\ddagger}$, J. Szmigielski }$^{\star}$
\end{large}
\\
\bigskip
\begin{small}
$^{\dagger}$ {\em CRM,
Universit\'e de Montr\'eal and Department of Mathematics and
Statistics, Concordia University}\end{small}

\begin{small}
$^{\ddagger}$ {\em Department of Mathematics, University of Notre Dame}\end{small}

\begin{small}
$^{\star}$ {\em Department of Mathematics and Statistics, University of Saskatchewan}\end{small}
\end{center}

\begin{center}{\bf Abstract}\end{center}
We apply the general theory  of Cauchy biorthogonal polynomials  developed in \cite{Bertola:CauchyMM} and \cite{Bertola:CBOPs} to the case associated with Laguerre measures.  In 
particular, we obtain explicit formul\ae\ in terms of Meijer-G functions for all key objects relevant to the study of the corresponding biorthogonal polynomials and the Cauchy two-matrix model associated with them. The central theorem we prove is  that a scaling limit of the correlation functions for eigenvalues near the origin exists,  and is given by a new determinantal two--level random point field, {\em the Meijer-G random field}. 
We conjecture that this random point field leads to a novel universality class of 
random fields parametrized by exponents of Laguerre weights.  
We express the joint distributions of the smallest eigenvalues in terms of suitable Fredholm determinants and evaluate them numerically.  We also   show that in a suitable limit,  the Meijer-G random field converges to the Bessel random field and hence the behavior of the eigenvalues of one of the two matrices converges to the one of the Laguerre ensemble.

\tableofcontents

\section{Introduction}

 The Cauchy two-matrix model, introduced in \cite{Bertola:CauchyMM}, is a random matrix  model defined in terms of a probability measure on the space of pairs $M_1, M_2$ of $n\times n$ {\em positive semidefinite} Hermitean matrices.  This probability  measure 
depends on the choice of two scalar functions $V_1,V_2: \R_+ \to \R$, 
called the {\em potentials}, and is defined as 
\be
\d\mu(M_1,  M_2) = \frac 1{\mathcal Z_n} \frac { {\rm e}^{-N\tr (V_1(M_1) + V_2(M_2))}}{\det (M_1+ M_2)^n} \d M_1\,  \d M_2,
\ee
where $\d M = \prod_{i<j} \d \Re M_{ij} \d \Im M_{ij} \prod_{k} \d M_{kk}$ stands for the ordinary Lebesgue measure on the real vector space of Hermitean matrices. The potentials $V_j$ are supposed to grow so that $\ds \liminf_{x\to+\infty} \frac {V_j(x)}{\ln x} = + \infty$.
The parameter $N$ is a  scaling parameter which in the asymptotic regime $n\to \infty$ tends  to infinity in such a way that $\frac n N \to T\in \R_+$. We will assume henceforth $T=1$ and that $N=n$.

There are several Hermitean multi matrix models; the most studied, and possibly the first, was introduced in \cite{EynardMehta};   the interaction, instead of $\det(M_1+M_2)^{-n}$,  is ${\rm e}^{-n \tr (M_1M_2)}$ which we will refer to as the ``Itzykson-Zuber'' (IZ) interaction.  Both models have applications to the counting of colored ribbon graphs on Riemann surfaces. The IZ models are expected to display new universality behaviours in appropriate scaling regimes. Partial results supporting that expectation are appearing (i.e. \cite{Duits:2012fk}, where the authors compute the scaling behavior of the kernels near special points of transition).
We briefly remark that one natural way of generating the $\det(M_1+M_2)^{-n}$ interaction is to consider
the measure 
${\rm e}^{-N\tr (V_1(M_1) + V_2(M_2))}
 {\rm e}^{- \tr A (M_1+M_2) A^\dagger}
\d M_1 \d M_2 \d A \d A^\dagger, 
$
where $A\in Mat(n\times n,\C)$ and $\d A \d A^{\dagger}$ is the standard Lebesgue measure 
on the set of complex matrices $A$, and to integrate out the Gaussian part ${\rm e}^{- \tr A (M_1+M_2) A^\dagger}$.  

It was shown  in \cite{Bertola:CauchyMM}, using methods  of \cite{EynardMehta},  that  all correlation functions of the eigenvalues of the two matrices $M_1, M_2$ can be computed exactly  in terms of certain {\em Cauchy biorthogonal polynomials} (BOPs).
The latter consist of two sequences of polynomials $\{p_n(x), q_n(y)\}_{n\in \N}$ of exact degree $n$ with the defining properties 
\be
\int_{\R_+^2} \frac { {\rm e}^{-N(V_1(x) + V_2(y))}}{x+y} p_\ell(x) q_m(y)\d x\,  \d y = \delta _{\ell m} \ ,\ \ p_n(x) = c_n x^n + \dots, \ \ \ q_n(y) = c_n y^n + \dots. \ \  \ c_n >0.
\label{12}  
\ee

In \cite{Bertola:CBOPs} the algebraic properties of these polynomials were investigated but no concrete example which could be considered ``classical'' was provided. 
On the other hand, even before the Cauchy BOPs were introduced, an instance reducible to  such polynomials and  associated with a classical weight, appeared implicitly in  \cite{Borodin:Biorthogonal}  in the study of a (different) biorthogonal Laguerre ensemble, one of  several examples
of biorthogonal ensembles considered there  that allow an explicit computation of correlation functions.
 In this paper, we apply the formalism developed in \cite{Bertola:CauchyMM} and \cite{Bertola:CBOPs} to the model defined by 
the probability measure (the factor of $n$ has been absorbed by an obvious rescaling)
\be
\d \mu(M_1,M_2)  = \frac 1{\mathcal Z_n} \frac{\det(M_1)^a \det (M_2)^b {\rm e}^{-\tr(M_1 + M_2)}}{\det(M_1+M_2)^n} \, \d M_1\, \d M_2,  
\label{LagCMM}
\ee
with associated  biorthogonal polynomials defined by
\be
\int_{\R_+^2} \frac {x^a y^b {\rm e}^{-x-y}}{x+y} p_\ell(x) q_m(y) = \delta _{\ell m} \ ,\ \ p_n(x) = c_n x^n + \dots, \ \ \ q_n(y) = c_n y^n + \dots. \ \  \ c_n >0.   \label{CBOPdef}
\ee

The present paper has three main goals: 
\begin{enumerate}
\item obtain explicit formul\ae\ for $p_n, q_n$ and related functions;
\item find explicit formul\ae\ for the correlation functions at finite $n$;
\item formulate a scaling limit of the correlation functions near the origin and thus define a limiting random point field; because of their expressions in terms of Meijer-G functions, we call this the Meijer-G random point field.
\end{enumerate}

In one application of the formalism developed in this paper  we express the {\em joint statistics}  of the fluctuations of the smallest eigenvalues of the two matrices $M_1,M_2$ in terms of a Fredholm determinant (Sec. \ref{applications}). This is followed by a numerical evaluation and plots of the distributions of the smallest eigenvalues.
We also perform a simple probe into how the Cauchy-Laguerre two-matrix model 
relates to the Laguerre ensemble.  To this end we formulate a suitable scaling limit in which we recover the Bessel field, thus showing that in an appropriate regime the spectrum of one of the matrices behaves like the spectrum of the Laguerre ensemble. 

\br
We point out that this is the first instance of a coupled matrix model for which one can address directly and rigorously  the {\em coupled} statistics of eigenvalues in a scaling regime: the IZ multimatrix model is -to date- far from this level of detail, hampered by the lack of an effective description of all four kernels.  As a result only the spectrum of one of the two matrices can be effectively analyzed \cite{Duits:2012fk} .
\er 

For the sake of comparison we briefly review the pertinent results for 
one-matrix models \cite{MehtaBook} using, as a prototype, the Gaussian Unitary Ensemble (GUE) with  probability measure $\d\mu(M) = \frac 1{Z_n} {\rm e}^{-\frac {n}2 \tr (M^2 )} \d M$ on a space of Hermitean matrices $M$ of size $n$. 
Let $\lambda_{max}$ denote the largest eigenvalue of $M$: in the limit $n\to\infty$ the probability that $\lambda_{max}>\sqrt{2}$  is zero. 
The fluctuations around this maximum in terms of the rescaled eigenvalues $x_i  = \sqrt{2} n^\frac 23 (\lambda_i- \sqrt{2})$ are known to be expressible in terms of a determinantal random point field (DRPF, see the review in Section \ref{review}) with the Airy kernel \cite{TracyWidomLevel}
\be
K_{\Ai}(x,y):= \frac {\Ai(x)\Ai'(y)- \Ai'(y)\Ai(x)}{x-y}.  
\ee
The famous Tracy-Widom result \cite{TracyWidomLevel} connected the  probability that $x_{max}<s$   to a special solution (Hastings-McLeod) of the second Painlev\'e\ equation as follows
\begin{align*}
\lim_{n\to\infty} Prob\le(\lambda_{max} \leq \sqrt{2} + \frac {\sqrt{2}s}{2n^{\frac 23}} \ri) =  
Prob(\hbox{no $x_i$'s in } [s,\infty)) = {\rm e}^{-\int_s^\infty (x-s) q(x)^2 \d x }\\
q''(s) = s q(s) + 2q(s)^3\ ,\ \ \ q(s) \sim \Ai(s)\ , \ \ s\to+\infty.  
\end{align*}
This behaviour is now known \cite{DKMVZ} to be {\em universal}, meaning that the Airy kernel arises in a similar scaling limit near the edge of the support of the limiting distribution of eigenvalues, for a general class of  potentials $V(M)$ instead of just $M^2$.  
Moreover, it is known that the Airy DRPF describes a generic behavior near a ``soft edge''.  
The Laguerre ensemble $\d\mu(M)= \frac 1{Z_n} (\det M)^a {\rm e}^{-\tr M}$ ($M$ positive definite) possesses a ``hard edge'' at the origin of the spectrum (zero eigenvalue). The statistics of the {\em smallest} eigenvalues is determined by the Bessel DRPF near the origin  and the gap probability is related to the third Painlev\'e\ equation \cite{Tracy-Widom-Bessel}. This behaviour is  also ``generic'' in the sense that it is stable under small perturbations and occurs whenever a hard-edge in the one-matrix model is present. 


%

The Cauchy-Laguerre two-matrix model we are considering in this paper is the benchmark for the behavior of a {\em coupled} random matrix model with a hard-edge and thus plays the same role as the Laguerre ensemble in relation to one-matrix models. We shall see that not only  can the model be completely elucidated  in terms of special functions, but also its scaling behaviour near the hard edge can be expressed in terms of  a DRPF as in \eqref{1-12}, with kernels described in terms of the generalized hypergeometric functions of type $_2F_2$, 
somewhat reminiscent of kernels and gap probabilities considered in \cite{BorodinDeift}.

In  \cite{BGSasympt} we have shown that  the spectrum of each of the matrices in the large $n$ limit leads to the same Airy-kernel  universality (or other universality classes that already appear in the one-matrix model) as long as the limiting eigenvalue distributions of the spectra does not contain the origin in its support:  therefore the largest (and smallest) eigenvalue distributions do not differ from the one-matrix case.

By contrast,  the model we study in this paper falls outside of that universality class: the limiting eigenvalue density was described in (\cite{BertolaBalogh}, Sec. 6) and near the origin it behaves like $x^{-\frac 23}$.  It is therefore natural to expect both new types of kernels as well as new types of gap distributions (see Remark \ref{remdens} and Fig. \ref{FigDensComp}).

The next section reviews the notions of a determinantal random point field, gap probabilities and their computation in terms of Fredholm determinants. 
Section \ref{summary} contains the results of computations involving special functions: the proofs of these results are in Sections \ref{setup} and \ref{CLBOPS}. The appendices contain further results of technical nature and some background material used in the main text.

\subsection{Short review of Determinantal Random Point Fields}
\label{review}
We review the fundamental notion  of {\em a random point field} (RPF) following \cite{Soshnikov-rev}. 
Let $X$ be a  topological  space, called {\em a configuration space}; in our case it shall be  $X=\R_+\sqcup \R_+$   equipped with the measure induced from the Lebesgue measure on each copy of $\R_+$ so that we can define $L^2(X)$. A {\em configuration} $\xi$ is  a locally finite collection of points of $X$.  A {\em random point field} on $X$ is a probability measure on the set of all configurations of points.  If $X$ is a disjoint union of $j$ sets we 
will call a random field  a $j$-level random point field. Given a Borel set $A\subset X$ we denote by $\sharp_A$ the integer-valued random variable counting the number of points in $A$. Given disjoint sets $A_1,\dots, A_m$  and a multi-index $\vec k\in \N^m$ one defines the {\em $k$-points correlation functions $\rho_k$ } by the formula ($\mathbb E$ denotes the expectation value)
\be
\mathbb E \prod_{i=1}^m \le(\frac{(\sharp_{A_i})!}{(\sharp_{A_i} - k_i)!}\ri) = \int_{A_1^{k_1} \times \cdots \times A_m^{k_m}} \rho_k(x_1,\dots, x_k) \d x_1\cdots \d x_k.  
\ee
The nontrivial fact  is that  the collection of correlation functions implicitly {\em defines} the probability measure on the space of all possible local configurations \cite{Soshnikov-rev}. 
 A (two-level) RPF on $X=\R_+\sqcup \R_+$ is  {\em a determinantal}  RPF (DRPF) if  all its correlation functions are determinants (see Definition 3' in \cite{Soshnikov-rev}) of the form 
\be
\rho_{(r,s)}(x_1,\dots, x_r; y_1,\dots, y_s) = \det\le[
\begin{array}{c|c}
\big[R^{(++)}   (x_i,x_j)\big]_{i,j\leq r} & \big[R^{(+-)}  (x_i, y_j) \big]_{i\leq r, j\leq s}\\[10pt]
\hline
\rule{0pt}{16pt}\big[R^{(-+)} (y_i,x_j)\big]_{i\leq s, j\leq r} & \big[R^{(--)}  (y_i, y_j) \big]_{i,j \leq s}  
\end{array}
\ri], 
\label{1-12}
\ee
where the functions $R^{(\pm \pm)}:\R_+^2\to \R$ are called ``kernels'' and  together they give rise to a single kernel $R: (\R_+\sqcup \R_+)^2 \to \R$. 
Thus {\em to define a DRPF it is sufficient to display its kernels}: we shall do this for both finite $n$ as well as for the scaling limit near the origin.

The eigenvalues of two positive definite matrices $M_1,M_2$ (which we denote by  $x_1,\dots,x_n$ and $y_1,\dots, y_n$) constitute an example of such DRPF. 
Given  a determinantal point field and a Borel subset $A$,  the associated "gap probability" is the probability that there are {\em no points} in $A$  and it is computed as follows. The kernel $R$ defines an integral operator $\mathcal R$ on $L^2(X)$.  Then the gap probability is given by (see \cite{Soshnikov-rev})
\be
Prob(\hbox{no points in } A) = \det \le(Id_{L^2(A)} -\pi_A \mathcal R \pi_A\ri), 
\ee
where $\pi_A:L^2(X)\to L^2(A)$ is the projection defined by restriction and the determinant is a Fredholm determinant.  

\section{The kernels for finite and infinite $n$: Meijer-G field}
\label{summary}
We recall the results of \cite{Bertola:CauchyMM}  (collected and explained in Appendix \ref{Correlapp}, in particular \eqref{eq:Hkernels}). The correlation functions of the eigenvalues in the Cauchy two-matrix model are expressed as determinants
\be
\rho_{(k,\ell)}(x_1,\dots, x_k;y_1,\dots,y_\ell) = \det\le[
\begin{array}{c|c}
{\ds H^{(n)}_{01}(x_i,x_j)\atop 
1\leq i,j\leq k} & 
{\ds H^{(n)}_{00}(x_i, y_j)
\atop 
1\leq i\leq k; 1\leq j \leq \ell}
\\[10pt]
\hline\\[1pt]
{\ds H^{(n)}_{11}(y_i,x_j) \atop 1\leq i\leq \ell ; 1\leq j \leq k}& 
{\ds H^{(n)}_{10}(y_i, y_j) \atop 
1\leq i,j\leq \ell}
\end{array}
\ri], 
\label{2-1}
\ee
where the kernels $H^{(n)}_{\rho\mu}$  are given by   (in the notation of \eqref{1-12})
\begin{subequations}\label{eq:Hs} 
\begin{align} 
R^{(++)} = H^{(n)}_{01}(x,y)& := x^{a } {\rm e}^{-x} K^{(n)}_{01}(x,y), &
R^{(+-)} = H^{(n)}_{00}(x,y) &:= x^{a }  y^{ b } {\rm e} ^{-x -y}  K^{(n)}_{00}(x,y), 
 \\
  R^{(-+)} =H^{(n)}_{11}(x,y) &:= K^{(n)}_{11}(x,y), & R^{(--)} = H^{(n)}_{10}(x,y) &:=  y^{ b } {\rm e}^{-y}  K^{(n)}_{10}(x,y),
\end{align}
\end{subequations} 
while the kernels $K^{(n)}_{\mu\nu}$ are defined in terms of the Cauchy biorthogonal polynomials $p_n,q_n$  as:
\begin{subequations}\label{eq:Ks}
\bea
&\ds K^{(n)}_{01}(x,y) := \int_{\R_+} \!\! K^{(n)}_{00}(x,y') \frac { y'^b {\rm e}^{-y'}\d y'}{y+y'}, &
K^{(n)}_{00}(x,y):= \sum_{j=0}^{n-1} p_j(x) q_j(y),
\nonumber 
\\
&\ds K^{(n)}_{11}(x,y) := \int_{\R_+^2} \!\! K^{(n)}_{00}(x',y') \frac { x'^a {\rm e}^{-x'}\d x'}{x+x'}\frac { y'^b {\rm e}^{-y'}\d y' }{y+y'}  - \frac 1{x+y} , &
K^{(n)}_{10}(x,y) := \int_{\R_+} \!\! K^{(n)}_{00}(x',y) \frac { x'^a {\rm e}^{-x'}\d x'}{x+x'}.  \nonumber 
\eea
\end{subequations}

Our first main result is a compact expression for these four kernels at finite $n$; 
to present it we set $\alpha=a+b$ and define two functions: 
\bea
&\& H_{c,n}(z) := 
\int_{\gamma} \frac {\d u }{2\pi i}  
\frac {\Gamma(u+c) \Gamma(n + \alpha-c  + 1 - u)}{\Gamma (1 - u)\Gamma(c+n+u)\Gamma(\alpha-c - u + 1)}z^{-u}, 
\label{Hn}
\\
&\& \wt H_{c,n}(w):=  \int_{ \gamma} \!\!\frac {\d u}{2\pi i} 
 \frac { \Gamma(u)\Gamma(u+c)\Gamma  \left( n+\alpha-c+1-u \right) }{ \Gamma(\alpha-c  + 1-u)  \Gamma  \left( c+n+u \right)} w^{-u}.  
  \label{wtHn}
\eea

In the above expressions 
$\gamma$ is a contour originating at $-\infty$ in the lower half-plane and returning to $-\infty$ in the upper half-plane in such a  way as to leave all the poles of the $\Gamma$ functions in the numerator containing the variable $+u$ inside the contour, while leaving those of the $\Gamma$ functions of variable $-u$ outside. Such types of integrals are Mellin-Barnes integrals and the expressions above are special cases of Meijer-G functions (see \cite{Bateman1}, 5.3, p. 206).  
Then 
\bt
\label{Hkernelthm}
The kernels $H^{(n)}_{\mu\nu}
$ (and thus the correlation functions \eqref{2-1}) are given by 
\begin{align*}
H^{(n)}_{01}(x,y) & = {\rm e}^{-x+y}\int_0^1 H_{a,n}(tx)\wt H_{b,n}(ty) \d t,&  
H^{(n)}_{00}(x,y)& ={\rm e}^{-x-y}  \int_0^1 H_{a,n}(tx)H_{b,n}(ty) \d t, 
\\
H^{(n)}_{11}(x,y)  & =   {\rm e}^{x+y}\int_0^1 \wt H_{a,n}(tx) \wt H_{b,n}(ty)  \d t  - \frac{{\rm e}^{x+y}}{x+y}, &
H^{(n)}_{10}(x,y) & ={\rm e}^{x-y}\int_0^1 \wt H_{a,n}(tx)H_{b,n}(ty) \d t, 
\end{align*}
while the kernels $K^{(n)}_{\mu\nu}$ are written in Theorems \ref{K00thm},\ref{K10thm},\ref{K11thm}.
\et
The above theorem is a summary of Theorems \ref{K00thm},\ref{K10thm},\ref{K11thm} which  contain computations of the kernels $K_{\mu\nu}^{(n)}$, whereas the ensuing expressions for the kernels $H_{\mu\nu}^{(n)}$ are obtained by a simple rewrite using the definitions \eqref{eq:Hs}, \eqref{eq:Ks}, \eqref{Hn}, \eqref{wtHn}, and the functions $G_{c,n}$ and $\wt G_{c,n}$ appearing in Theorems \ref{K00thm},\ref{K10thm},\ref{K11thm}. In particular,  $H_{c,n}(\zeta) = \zeta^c G_{c,n}(\zeta)$  and $\wt H_{c,n}(\zeta) = \zeta^c \wt G_{c,n}(\zeta)$.

\subsection{Scaling limit: the Meijer-G random point field}
In the limit $n\to \infty$, with the substitutions $ x:= \frac \zeta {n^2}\le(\frac {n}{n+1}\ri)^\alpha ,\ y:= \frac \xi {n^2}\le(\frac {n}{n+1}\ri)^\alpha$,  we arrive\footnote{The unusual scaling with the factor $(\frac {n}{n+1})^{\alpha}$ is used to have a higher order of approximation as $n\to\infty$.}  at  a novel universality class of random point fields that we name  the {\em Meijer-G random point field}.
\bd
\label{defHfun}
Let 
\bea
  H_{c}(\zeta)&\& :=
   \int_{\gamma} \frac {\d u }{2i\pi} \frac {\Gamma(u+c)}{\Gamma ( 1 - u)\Gamma(\alpha-c +1- u )}\zeta^{-u};
 \ \ \ \ 
 \wt  H_{c}(\zeta) :=
   \int_{ \gamma} \frac {\d u}{2i\pi} 
 \frac { \Gamma(u)\Gamma(u+c)}{ \Gamma(\alpha-c  + 1-u)} \zeta ^{-u}.
 \label{defwtH}
 \eea
 The contour $\gamma$ is a contour of the form in Fig. \ref{MeijerContour} enclosing all the poles in the numerators of the integrands.
 Note that $H_c, \wt H_c$ are Meijer-G functions as in \cite{Bateman1} and \href{http://dlmf.nist.gov/16.17} {16.17 in DLMF} 
\footnote{DLMF="Digital Library of Mathematical Functions", http://dlmf.nist.gov}
(see definition in App. \ref{MGapp}).
\ed

\bt[Meijer-G two-level  random point field and universality class]
\label{MeijerGKprime}
In the scaling limit  the correlations of the eigenvalues of $M_1,M_2$ are determined by the two--level random point field 
on the configuration space $\R_+\sqcup \R_+$ with the kernels below (in the notation of \eqref{1-12})
\bea
R^{(++)} = \lim_{n\to \infty}\frac 1{n^2}\le(\frac {n}{n+1}\ri)^\alpha H^{(n)}_{01}\le(x,y \ri) &\&  = \mathcal G_{01}(\zeta, \xi),
\quad 
R^{(+-)}= \lim_{n\to \infty}\frac 1{n^2}\le(\frac {n}{n+1}\ri)^\alpha H^{(n)}_{00} \le(x, y\ri) = \mathcal G_{00}(\zeta, \xi),
 \nonumber \\
R^{(-+)} = \lim_{n\to \infty}\frac 1{n^2} \le(\frac {n}{n+1}\ri)^\alpha H^{(n)}_{11}\le(x,y\ri)  &\& = \mathcal G_{11}(\zeta, \xi),\quad
R^{(--)}=
\lim_{n\to \infty}\frac 1{n^2}\le(\frac {n}{n+1}\ri)^\alpha H^{(n)}_{10}\le(x,y\ri) =  \mathcal G_{10}(\zeta, \xi), 
  \nonumber 
\eea
where  the points in the first copy of $\R_+$ will be called the "$+$" field, and the others the "$-$" field.   The kernels are
\bea
 \mathcal G_{01}(\zeta,\xi)  = \int_0^1 H_{a}(t\zeta) \wt H_{b}(t\xi) \d t,\qquad
 \mathcal G_{00} (\zeta,\xi) =  \int_0^1 H_{a}(t\zeta) H_{b}(t\xi) \d t, 
\label{g00g01} \\
 \mathcal G_{11}(\zeta,\xi)  = 
\int_0^1 \wt H_{a}(t\zeta) \wt H_{b}(t\xi ) \d t  - \frac{1}{\zeta+\xi}, 
\qquad 
 \mathcal G_{10}(\zeta,\xi) =  \int_0^1 \wt H_{a}(t\zeta) H_{b}(t\xi) \d t, 
\label{g10g11}
\eea
with $H_c,\wt H_c$ as in Definition \ref{defHfun}. The convergence is uniform for $\xi, \eta$ within compact sets and the error of the approximation is within $\mathcal O(n^{-2})$. 
\et
 Proposition \ref{propConcomitant} provides alternative expressions for the kernels $\mathcal G_{\mu\nu}$ in terms of ``point-split bilinear concomitants'', involving no integration, only derivatives.
 Section \ref{MGrpf} is devoted to the proof of Theorem \ref{MeijerGKprime}.
We expect the following conjecture to be true.  
\begin{conjecture}\label{conj:Meijer}
The Meijer random field obtained in the scaling limit in this paper 
is universal within the class of Cauchy matrix models  of the form 
\begin{equation*}
\d \mu(M_1,M_2) =\frac 1{\mathcal Z_n} \det (M_1)^a \det (M_2)^b \frac {{\rm e}^{-n \tr \left (V_1(M_1)+ V_2(M_2)\right)}}{\det(M_1 + M_2)^n}\d M_1 \d M_2
\end{equation*}
with $V_i$ analytic near the origin and the scaling $x\mapsto x\,n^{-3}$.
 \end{conjecture}
\section{Applications}
\label{applications}
The simplest statistical information is the density of eigenvalues both for finite $n$ and in the scaling limit, in either case obtained directly from the kernels; for the first matrix (similar expression holds for the second matrix)
\be
\rho_1^{(n)} (x) =H^{(n)}_{01} (x,x) =  \int_0^1 t^{a+b} 
G^{1,1}_{2,3} \left(\left.{-\alpha -n,n\atop 0,-a, -\alpha}\; \right | t x\right) 
G^{2,1}_{2,3} \left(\left.{-\alpha -n,n\atop 0,-b,-\alpha}\; \right | tx \right) \d t
\ee
which follows from the expression in Theorems \ref{Hkernelthm}, \ref{K00thm}, \ref{K10thm} (see \eqref{G}, \eqref{wtG}).
For large $n$ (in fact even for small $n$'s) and $x = \frac \zeta {n^2} \le(\frac {n}{n+1}\ri)^{\alpha}$
\be
\frac 1 {n^2}\le(\frac {n}{n+1}\ri)^{\alpha} \rho_1^{(n)}\le(x\ri) \m{=}^{\tiny {Corollary \ref{MeijerGH}}} 
\int_0^1 G^{1,0}_{0,3} \left(\left.{\atop a,0, -b}\; \right | t\zeta  \right)
G^{2,0}_{0,3} \left(\left.{\atop b,0, -a}\; \right | t\zeta \right)\d t + \mathcal O(n^{-2}) \label{goodapp}
\ee
A more effective formula is obtained from Proposition \ref{propConcomitant}, which involves only derivatives (the expression is cumbersome, so we have opted here  for the integral expression instead).
Figure \ref{figdensity} compares the exact density (solid line) with the asymptotic density as per \eqref{goodapp}.
\begin{figure}[t]
\includegraphics[width=0.33\textwidth]{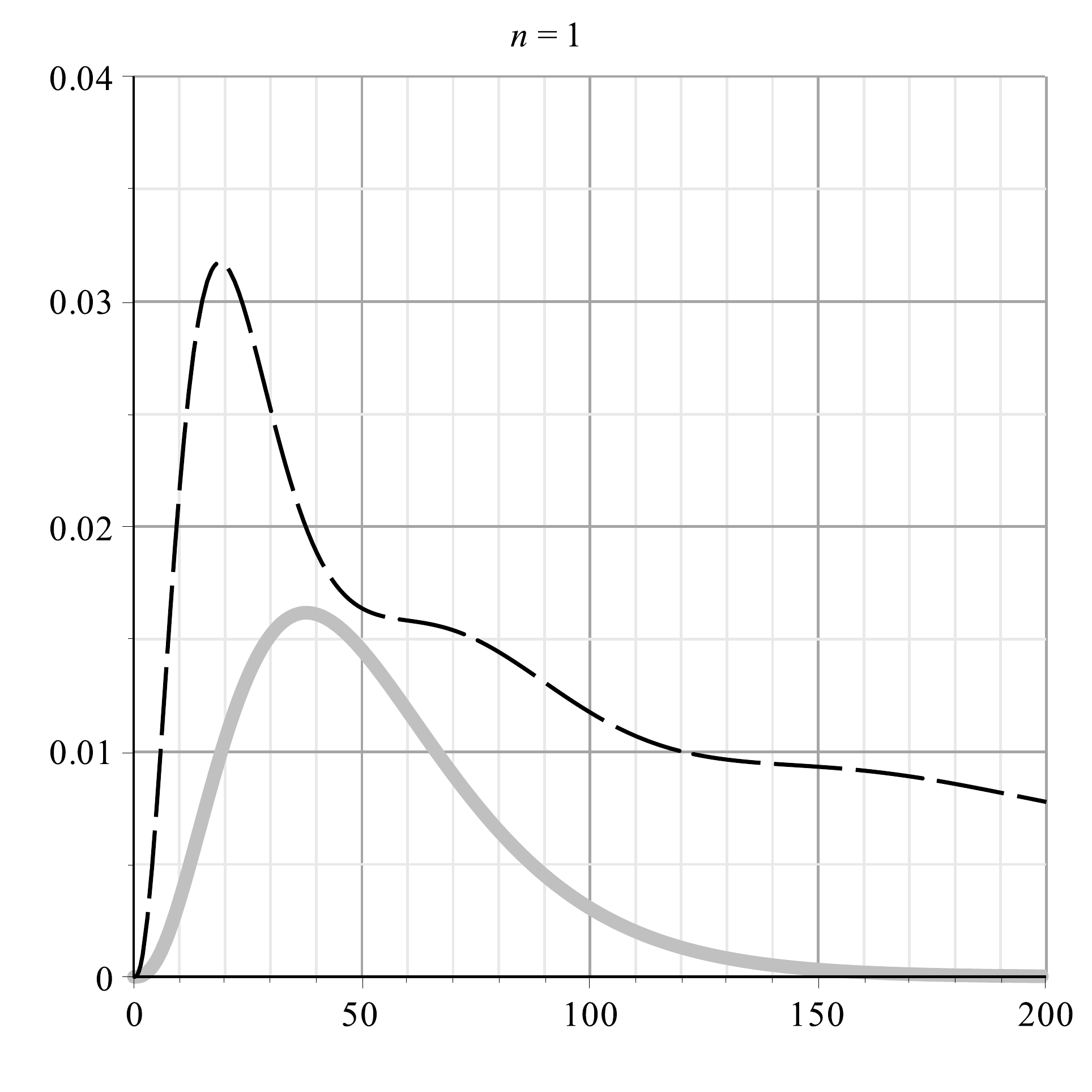}
\includegraphics[width=0.33\textwidth]{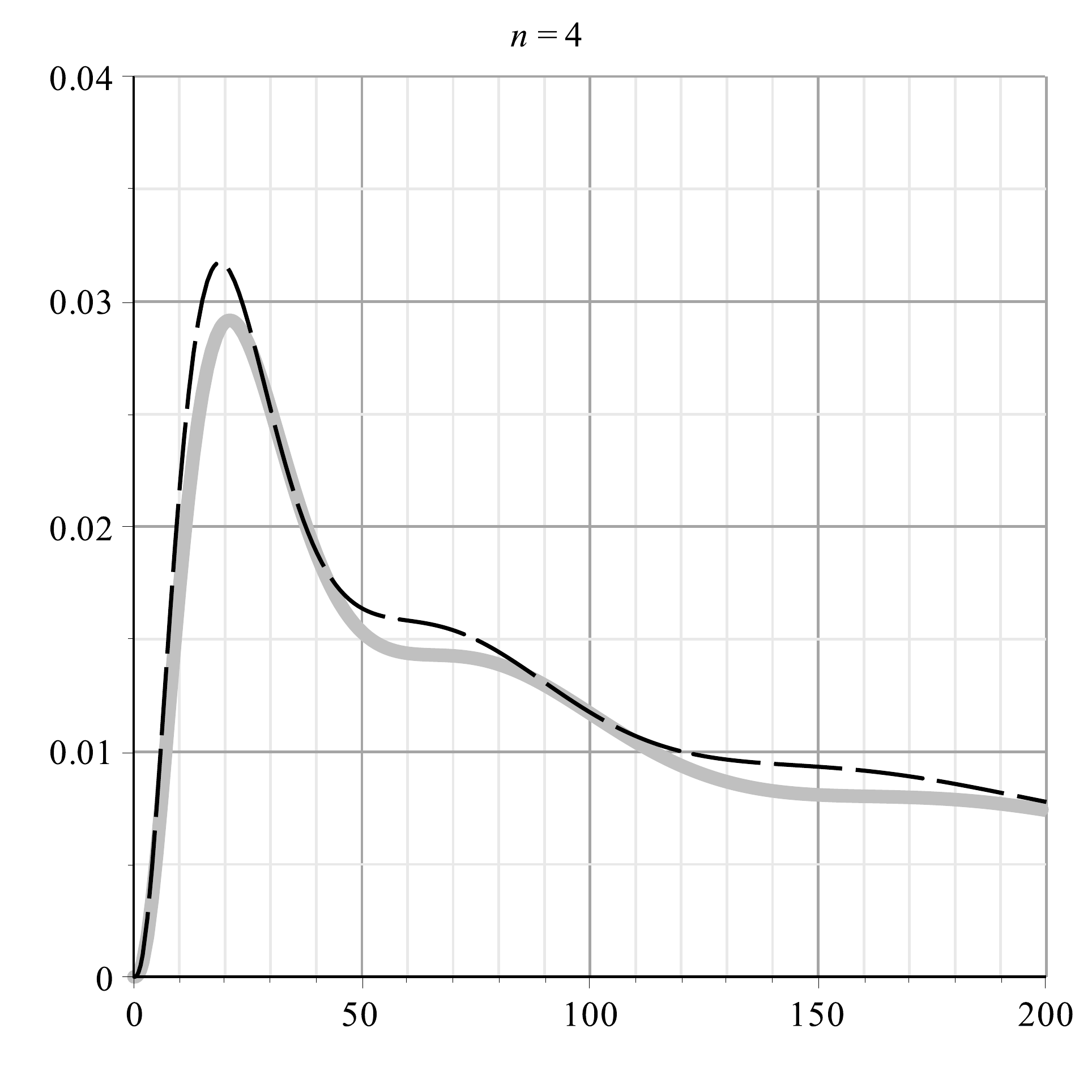}
\includegraphics[width=0.33\textwidth]{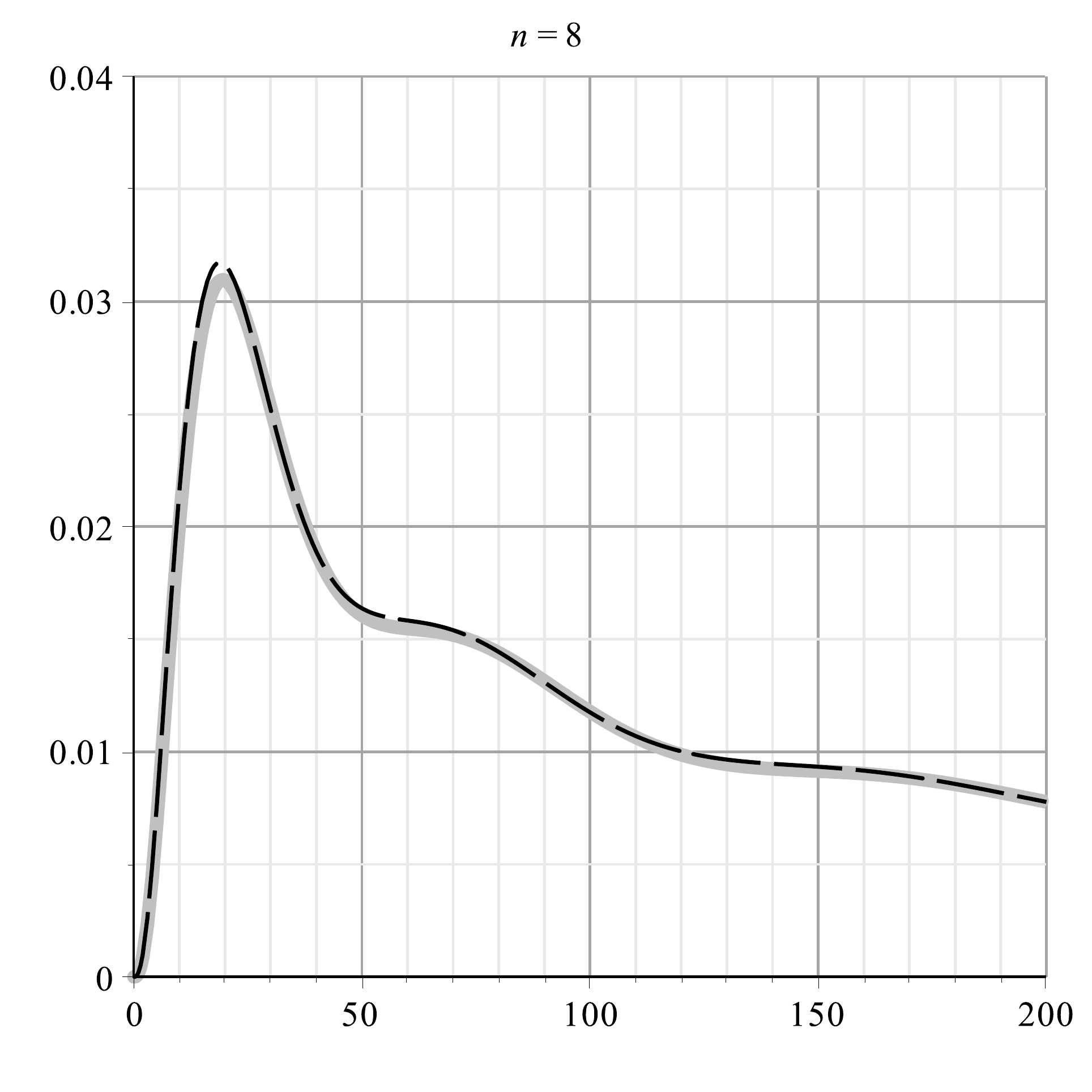}
\caption{The exact density for $n=1,4,8$ versus the asymptotic formula \eqref{goodapp}. Here $a = 3, b=1$  }
\label{figdensity}
\end{figure}
\subsection{The distribution of the smallest eigenvalues}
According to the general theory outlined in Section \ref{review}, the probability that the smallest eigenvalue of $M_j$ is greater than some $x>0$ can be expressed in terms of a Fredholm determinant. 
Denote by $\mathbb P_n$ the probability measure \eqref{LagCMM} on the space of pairs $M_1,M_2$ of positive-definite Hermitean matrices of size $n\times n$.
We give here two examples.
\paragraph{Probability that $\bf spect (M_1)\boldsymbol{\subset \le[\frac s{n^2},\infty\ri)}$.}
\begin{wrapfigure}{r}{0.4\textwidth}
\vspace{-20pt}
\includegraphics[width=0.3\textwidth]{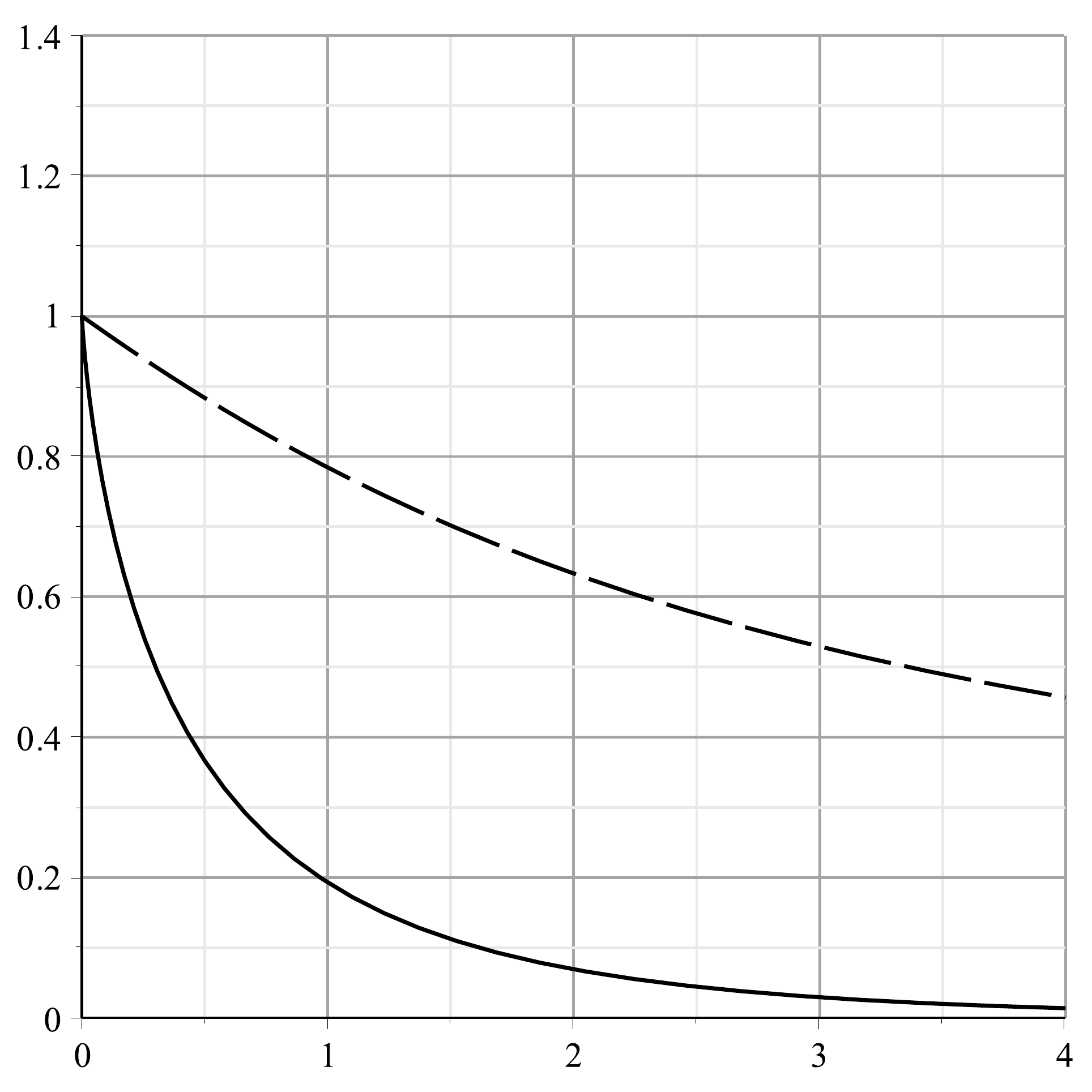}
\caption{Comparison between the gap probabilities of the Meijer-G (solid) field and the Bessel field (dashed).}
\label{MeijerBessel}
\end{wrapfigure}
Denote by $0\leq  x_1\leq x_2\leq \dots$ the eigenvalues of $M_1$ and $0\leq y_1\leq y_2\leq \dots$ those of $M_2$. 
Then our results on the scaling limits of the kernel imply the following
\bea
\lim_{n\to \infty} \mathbb P_n\le(x_1> \frac s{n^2} \ri) = \det\le[Id_{L^2([0,s])} - \mathcal R_{++} \ri]=:F_1^{(1)}(s)
\label{F1}
\eea
where $\mathcal R_{++}$ is the integral operator with the kernel $R^{(++)}(\zeta,\xi) = \mathcal G_{01}(\zeta,\xi)$ defined in \eqref{g00g01} restricted to the interval $[0,s]$, with a similar definition of $F_1^{(2)}$ (see footnote\footnote{The superscript $^{(1)}$ in $F_1^{(1)}$ refers to the matrix $M_1$, with a similar definition for $F_1^{(2)}(t):= \lim_{n\to\infty}\mathbb P_n\le(y_1> \frac t{n^2} \ri) $.}).

\begin{figure}[t]
\includegraphics[width=0.33\textwidth]{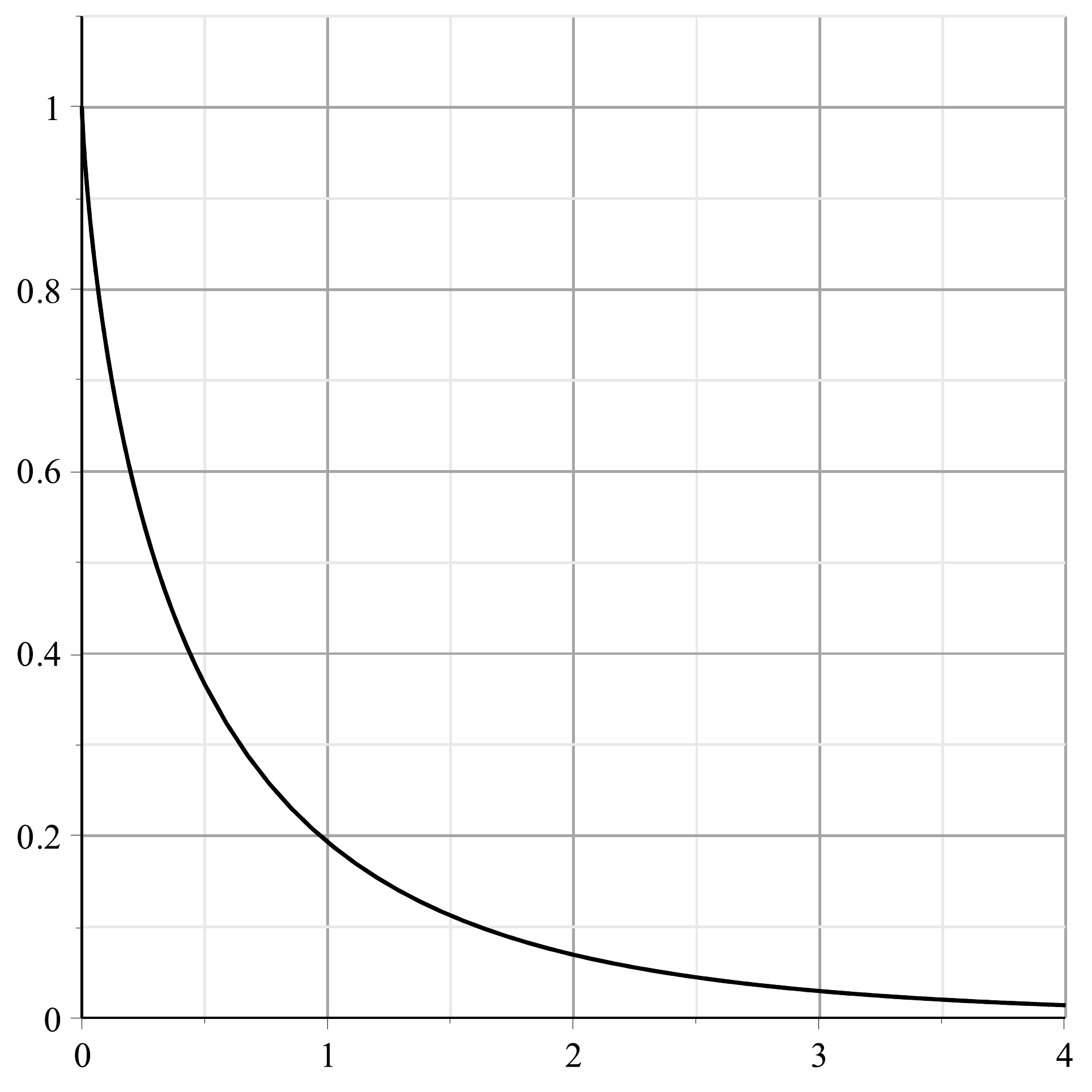}
\includegraphics[width=0.33\textwidth]{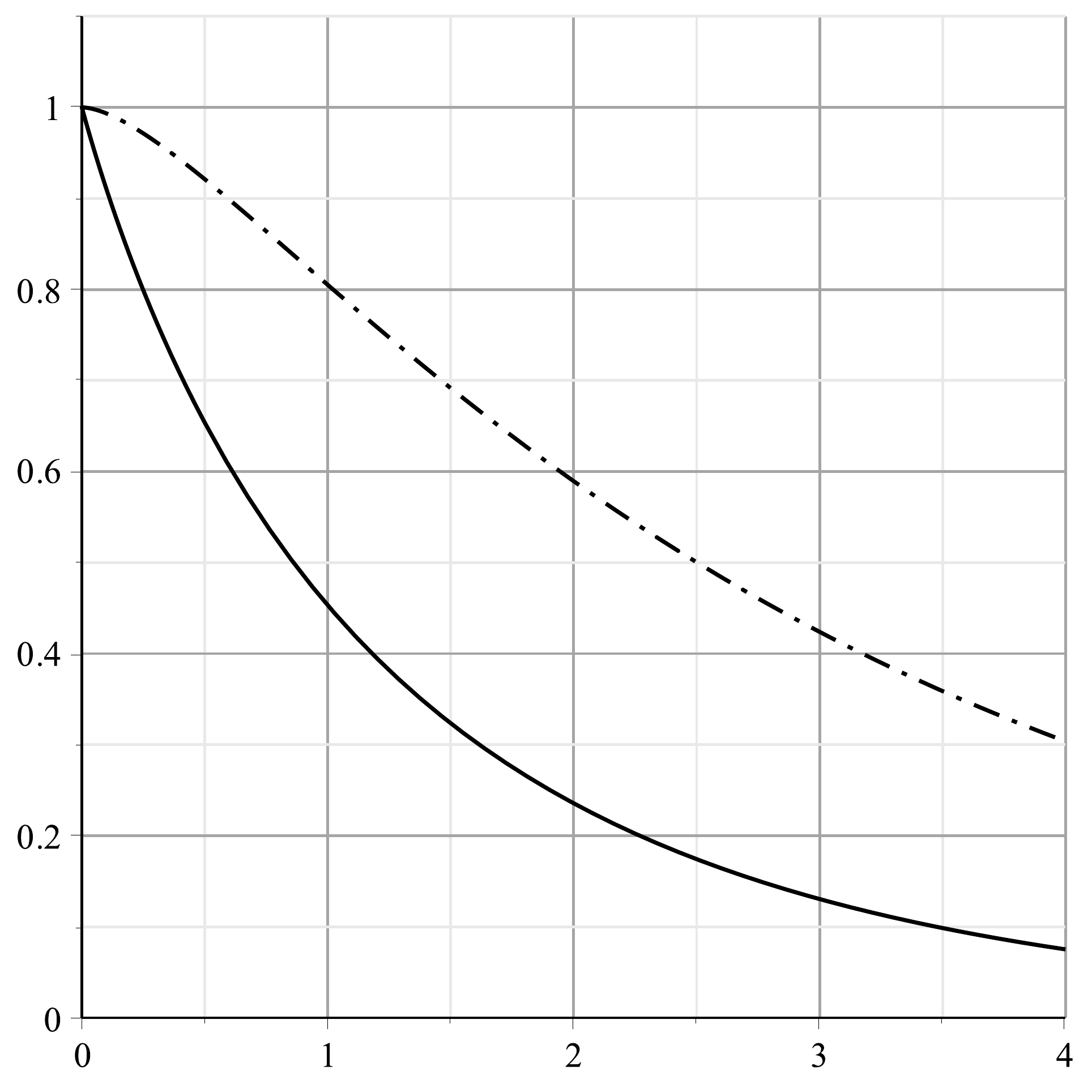}
\includegraphics[width=0.33\textwidth]{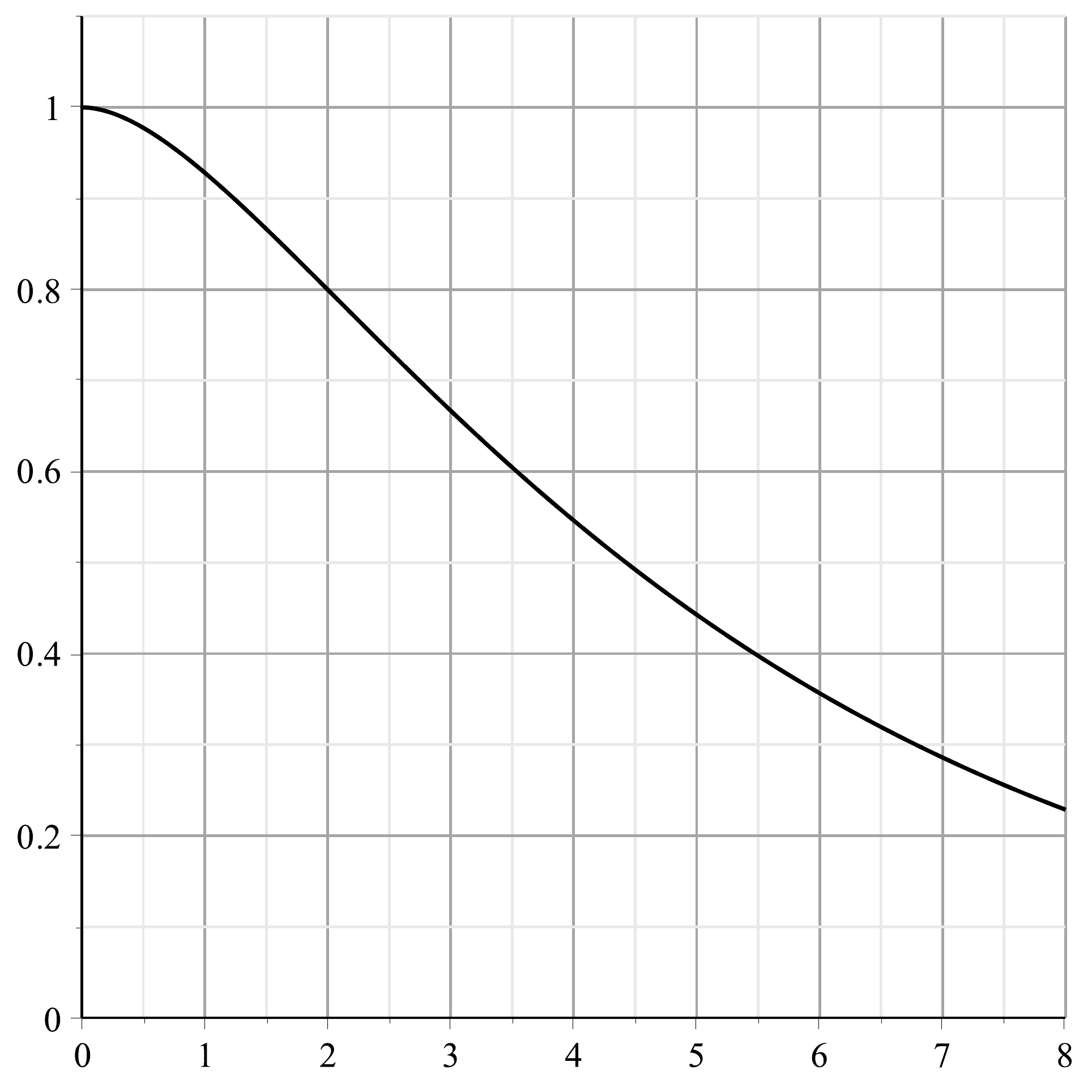}
\caption{The gap probability $F_1^{(1)}$ \eqref{F1} (and $F_1^{(2)}$ in dashed) computed numerically as explained in \cite{Bornemann}. From left to right,  $a = 0, b=0$,  $a = 0, b=1$ and $a=1,b=1$.}
\label{FigFred}
\end{figure}
\begin{wrapfigure}{r}{0.30\textwidth}
\vspace{-20pt}
\includegraphics[width=0.25\textwidth]{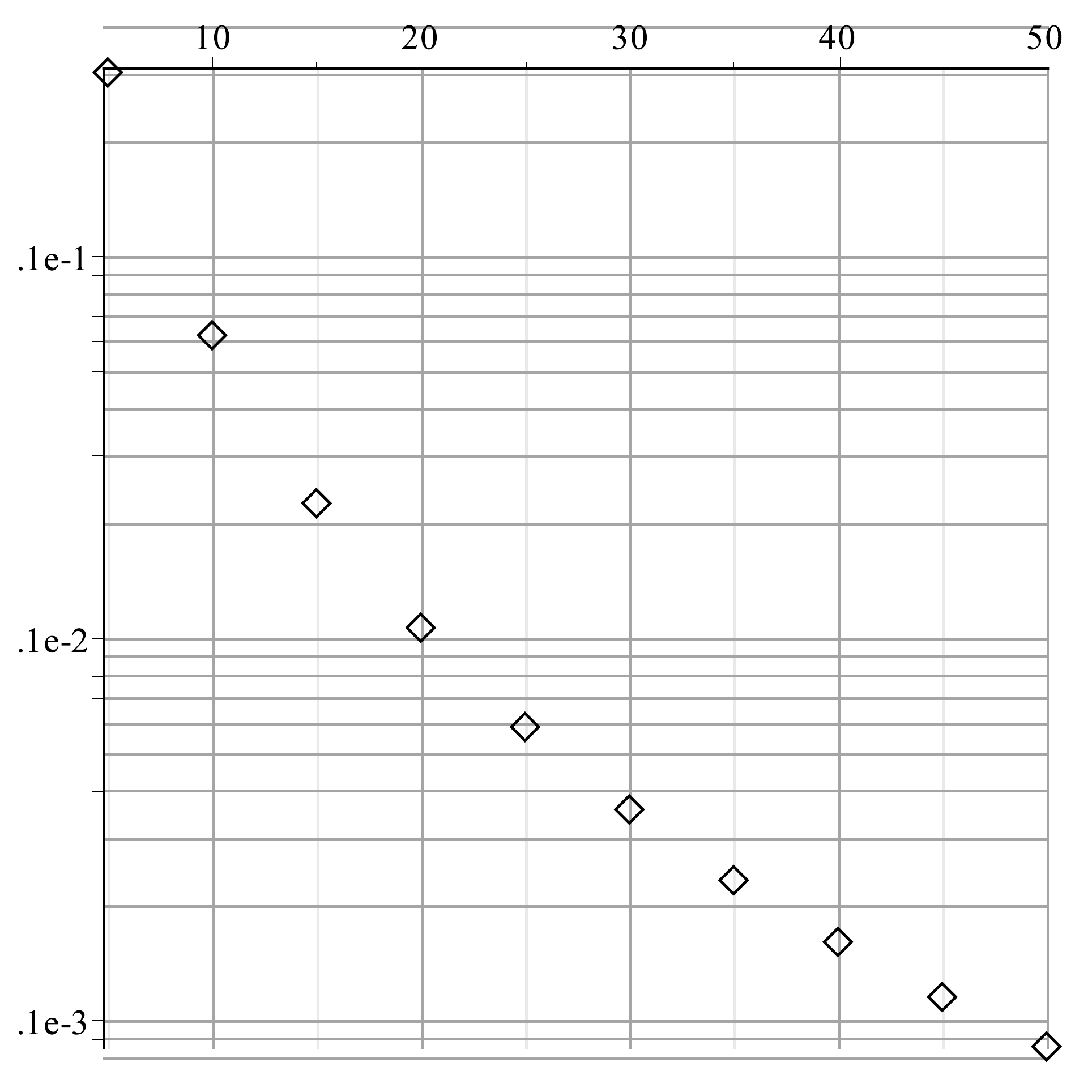}
\includegraphics[width=0.25\textwidth]{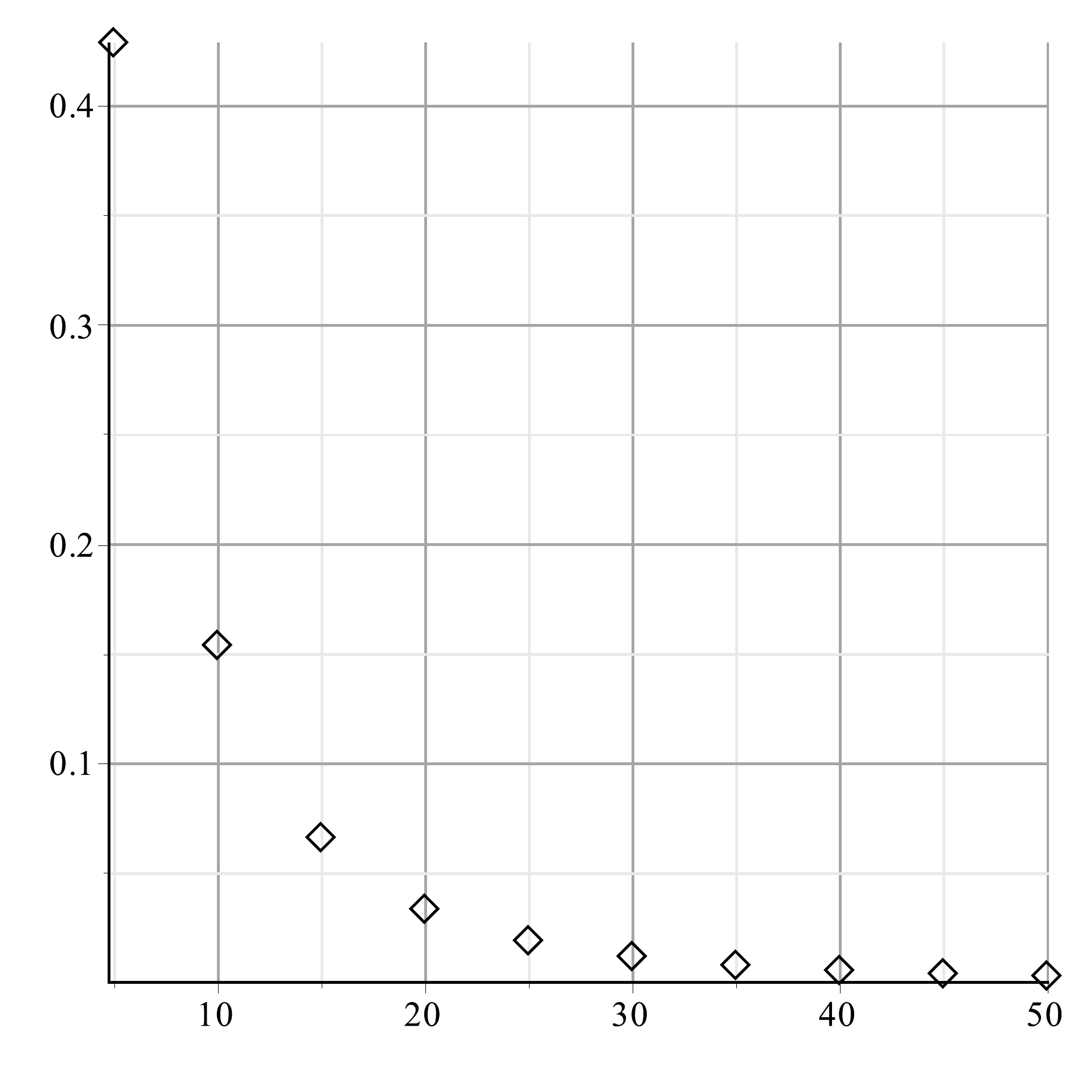}
\caption{ The absolute (top) and relative  errors of $F_1^{(1)}(3)$ with sizes of the discretized kernel on the horizontal axis. Since the kernels are smooth (in fact, analytic) the convergence is rather fast (see \cite{Bornemann} for more details). }
\label{FigFrederr}
\end{wrapfigure}
Numerical evaluation using the method in \cite{Bornemann} is shown in Figure \ref{FigFred}.
 In Figure \ref{MeijerBessel} we compare the Bessel field with $\alpha=0$ with the Meijer-G field with $a=0=b$. It is natural to compare these two since they both describe a scaling limit of a  random matrix  model with a hard edge at the origin.
 Considering their random matrix origin it is clear that the spectrum of the Meijer-G field is more attracted towards the hard edge  due to the effect of an attraction exerted by the eigenvalues of the other matrix. See also Section \ref{secMtoB} and Remark \ref{remdens}.
\paragraph{Probability that $\boldsymbol{ spect (M_1)\times spect(M_2) \subset \le[\frac {s}{n^2},\infty\ri)\times \le[\frac t{n^2},\infty\ri)}$.}
With the same notations as above for the eigenvalues of $M_1, M_2$, our results on the scaling limits of the kernel imply
\be
\lim_{n\to \infty} \mathbb P_n\le(x_1> \frac s{n^2}, y_1> \frac {t}{n^2}\ri) = \det\le[Id_{\mathcal H} - \mathcal R \ri]=:F_2(s,t), 
\label{F2}
\ee
where $\mathcal H= L^2((0,s)\sqcup (0,t)) \simeq L^2((0,s))  \oplus L^2((0,t))$ and $\mathcal R$ is the integral operator  on $\mathcal H$ defined in the introduction. 
Explicitly it reads as follows: denote by $\phi(\zeta) = \le[{\phi_+(\zeta)\atop \phi_-(\zeta)} \ri]$ a vector of $\mathcal H$, so that $\phi_+(\zeta) \in L^2((0,s))$ and $\phi_-(\zeta)\in L^2((0,t))$. Then 
\be
\mathcal R \phi (\zeta) = \le[
\begin{array}{c}
\ds \int_0^s \mathcal G_{01}(\zeta,\xi)\phi_+(\xi) \d \xi + \int_0^t \mathcal G_{00}(\zeta,\xi)\phi_-(\xi) \d \xi \\
\ds \int_0^s \mathcal G_{11}(\zeta,\xi)\phi_+(\xi) \d \xi + \int_0^t \mathcal G_{10}(\zeta,\xi)\phi_-(\xi) \d \xi 
\end{array}
\ri].  
\ee
It is difficult to gain a quantitative understanding  of the level of correlation between the two matrices; for this reason we have carried out a numerical computation showing, by the way of example, the quantity  $1 - \frac {F_1^{(1)}(s)F_1^{(2)}(t)}{F_2(s,t)}$, which, in view of their definition, would be identically zero if the spectra were independent (see Fig. \ref{FigFred2}).  
\begin{figure}[t]
\includegraphics[width=0.3\textwidth]{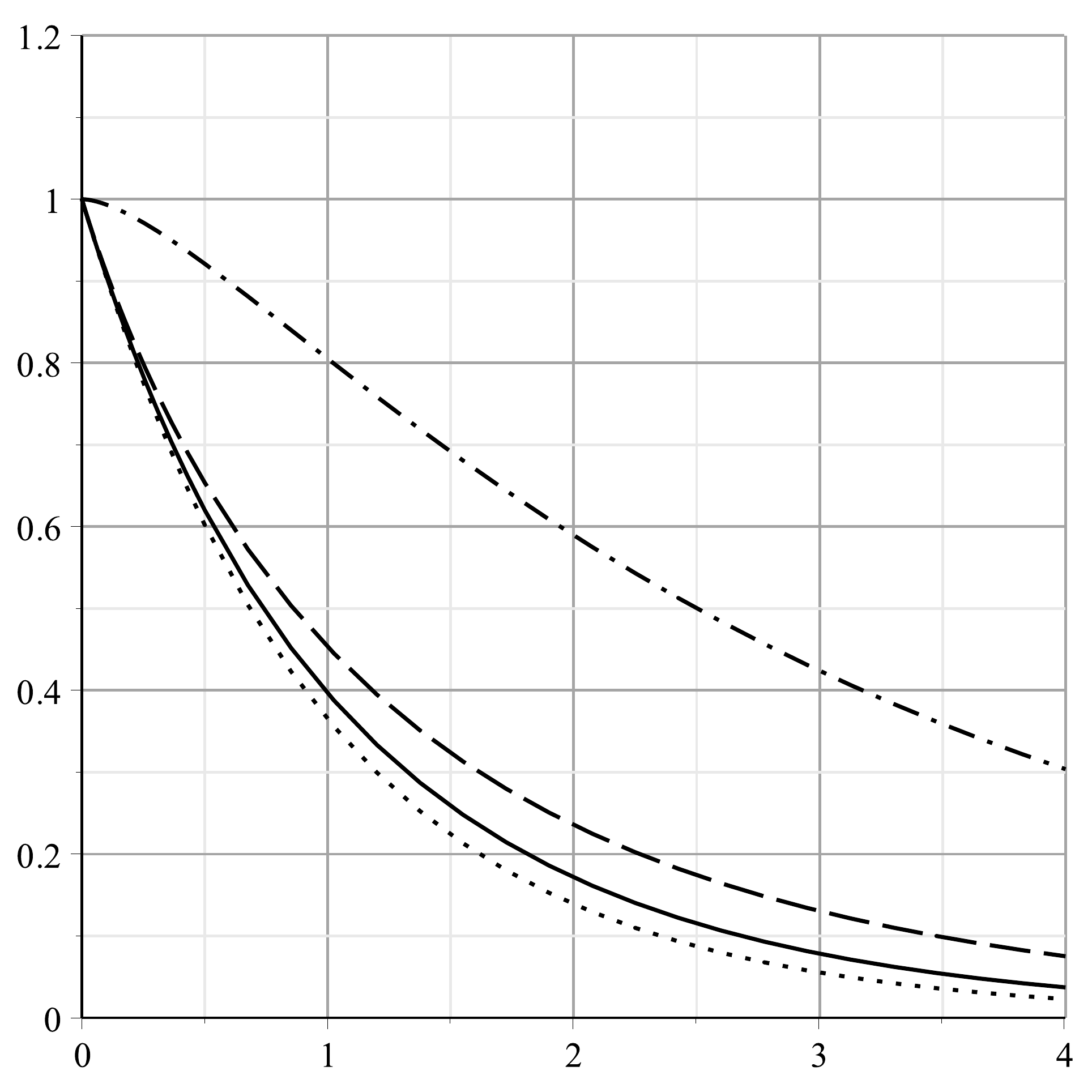}
\includegraphics[width=0.3\textwidth]{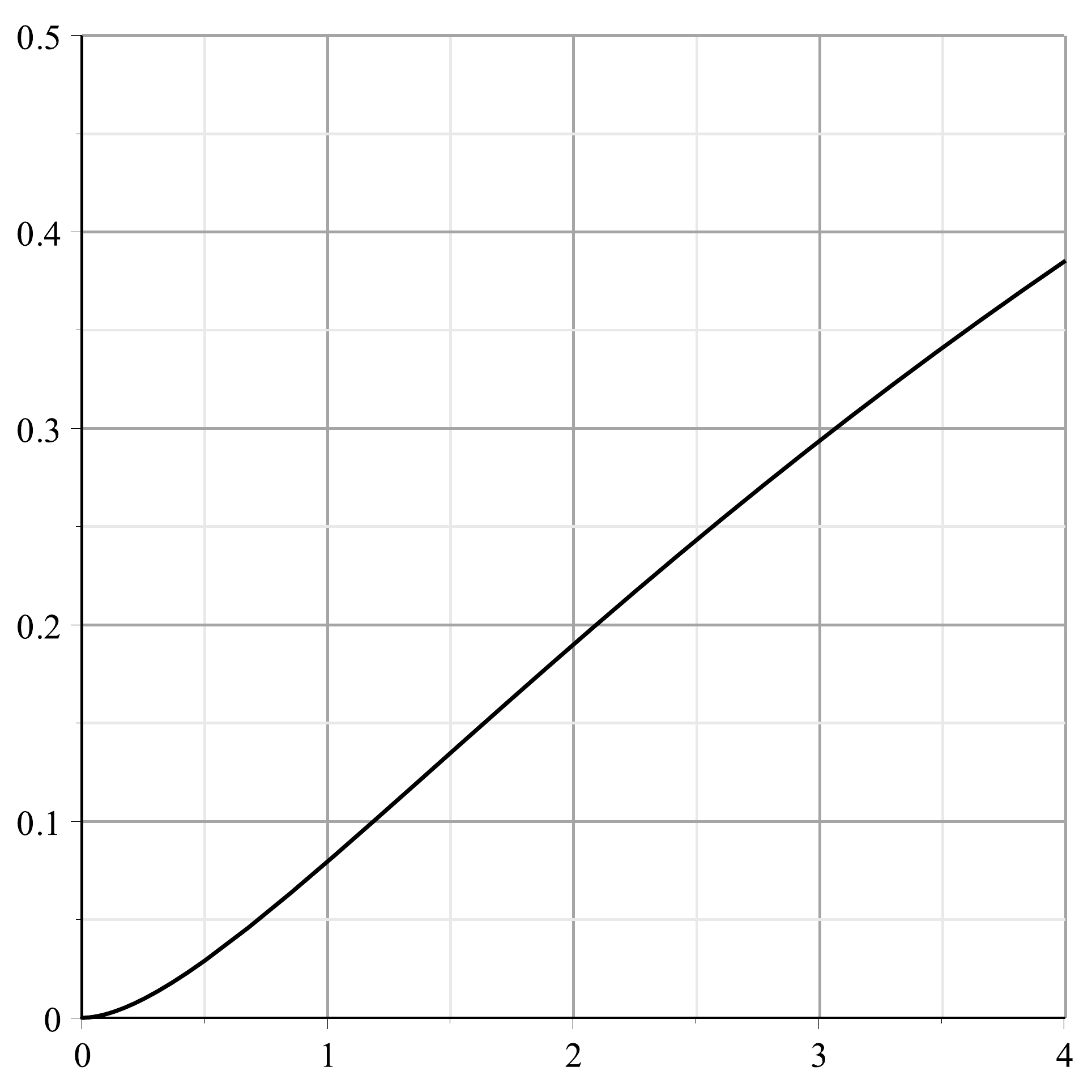}
\includegraphics[width=0.4\textwidth]{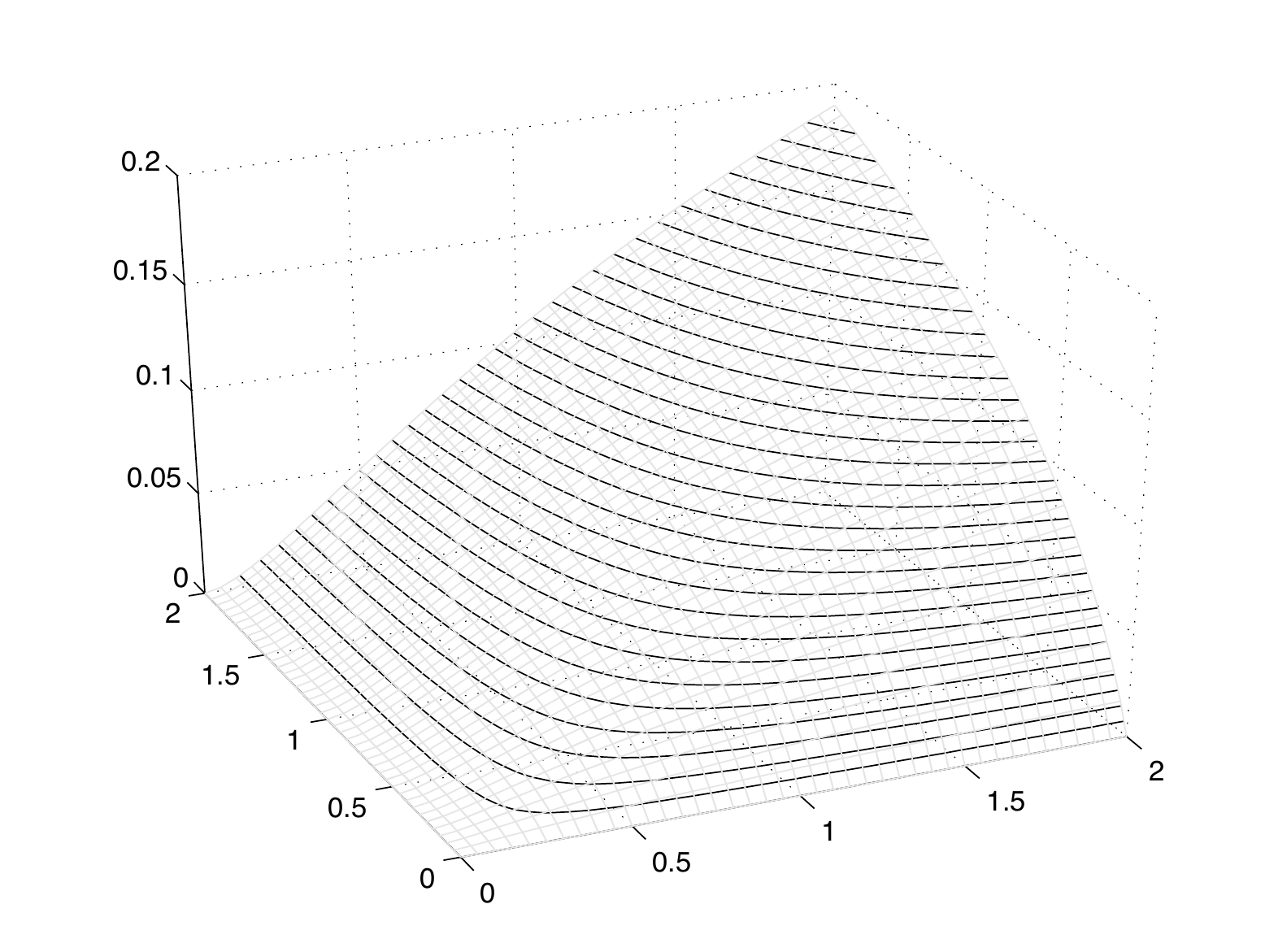}
\caption{The gap probabilities $F_2(s,s)$ (solid), $F_1^{(1)}(s)$ (dash), $F_1^{(2)}(s)$ (dash-dot) and $F_1^{(1)}(s) F_2^{(2)}(s)$ (dots) (\eqref{F2},\eqref{F1}) computed numerically as explained in \cite{Bornemann}, here for $a = 0, b=1$. In  the center the quantity  $1-\frac { F_1^{(1)}(s)F_1^{(2)}(s)}{F_2(s,s)}$ and on the right the plot of $1-\frac{F_1^{(1)}(s)F_1^{(2)}(t)}{F_2(s,t)}$ on the domain $[0,2]\times [0,2]$. This quantity is a measure of deviation from independence, and it would be identically zero if the spectra were independent.}
\label{FigFred2}
\end{figure}

\subsection{Convergence to the Bessel field}
\label{secMtoB}
Let us consider our Meijer-G DRPF defined by the kernels in Theorem \ref{MeijerGKprime}.  
\bt
\label{toB}
In the limit $b\to +\infty$  (and $a$ fixed) the $+$ field becomes the Bessel DRPF (with parameter $a$)  under the rescaling 
$\zeta\mapsto \frac b 4 \zeta$. Under the same rescaling the points of the  $-$ field (almost surely) do not occupy any bounded set.
\et
To prove the theorem it suffices to show that all the correlation functions involving the $-$ field tend to zero uniformly on compact sets, while 
the  correlation functions involving the $+$ field alone  become the correlation functions of the Bessel DRPF with the standard kernel \eqref{BesselKernel0}. Specifically 
\be
\lim_{b\to+\infty} \frac b 4 R^{(++)}\le(\frac b 4 \zeta  ,\frac b 4  \xi \ri)=
\lim_{b\to+\infty} \frac b 4 \mathcal G_{01}\le(\frac b 4 \zeta  ,\frac b 4  \xi \ri)= \le( \frac {\zeta}{\xi} \ri)^\frac a 2 K_{_B,a}(\zeta,\xi)
\label{gtob}
\ee

\begin{wrapfigure}{r}{0.3\textwidth}
\vspace{-20pt}\includegraphics[width=0.3\textwidth]{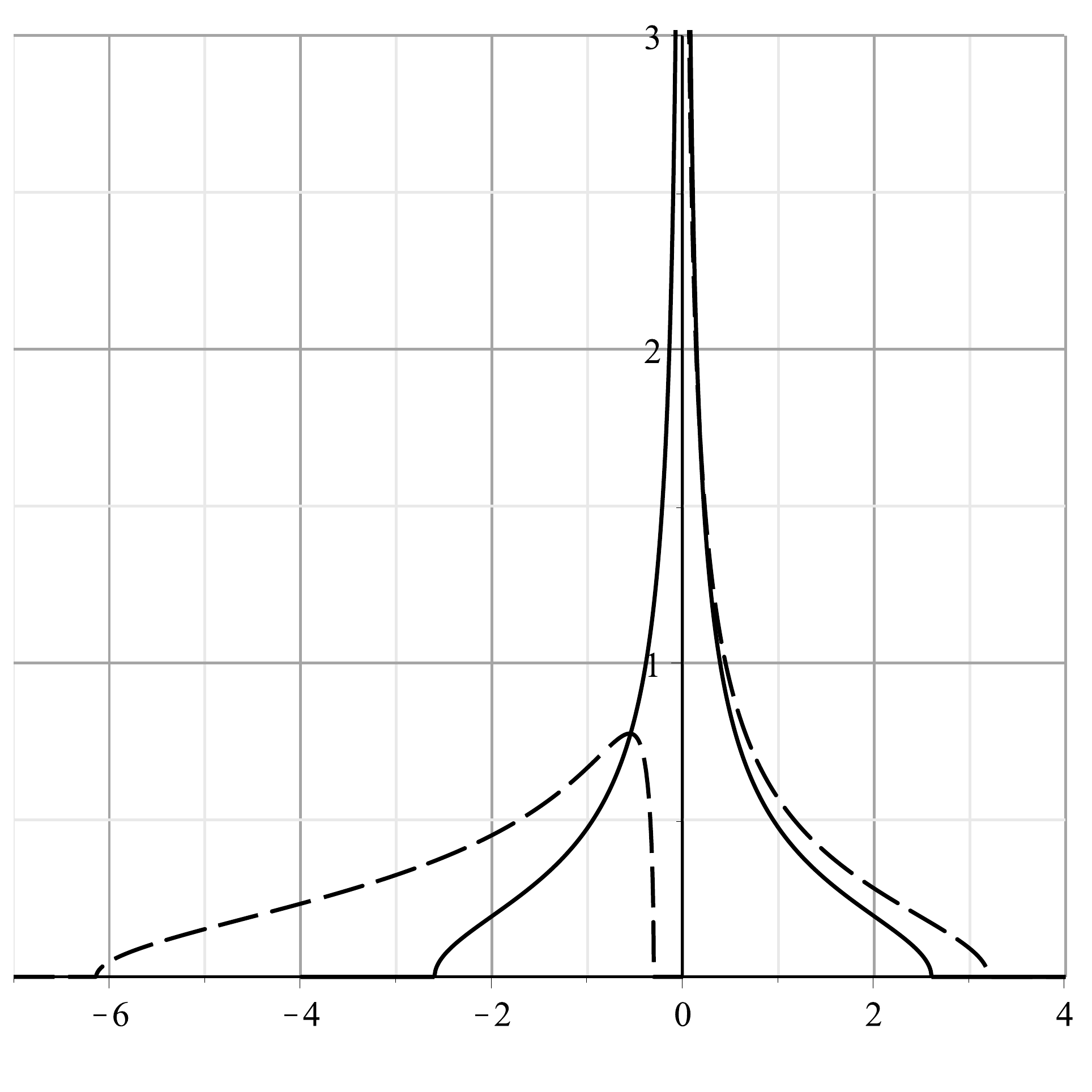}
\caption{The limiting density for the eigenvalues of $\frac 1 n M_1,\frac 1 n M_2$ (those of $M_2$ are plotted on the negative axis). Solid: $a, b$ do not scale with $n$. Dashed: $a$ does not scale, $b = \beta n$ (here $\beta = 2$).
 }
\label{FigDensComp}
\end{wrapfigure}
\noindent where the kernel $\mathcal G_{01}$ is defined in Theorem \ref{MeijerGKprime} (the dependence on $a, b$ of the kernel is not indicated explicitly but can be read off \eqref{g00g01}).  The correlation functions constructed from the determinants of the kernel in \eqref{gtob} are the same as the correlation functions of the Bessel kernel $K_{_B,a}$ because the prefactor $(\zeta/\xi)^{a/2}$ drops from all the determinants (it amounts to a conjugation of the matrix $R^{(++)}(\zeta_i,\zeta_j)$ by a diagonal matrix). The proof of \eqref{gtob} by a direct computation is included in Appendix \ref{MtoB}, where it is also shown that all the correlation functions involving the $-$ field tend to zero uniformly on compact sets of the (rescaled) variables (in particular the kernels $R^{(--)}$ and $R^{(+-)}$ tend to zero, and $R^{(-+)}$ has a limit,  which, however, does not affect the correlation functions).\par\vskip 4pt

 An immediate consequence of Theorem \ref{toB} is that the corresponding gap probabilities $F_1^{(1)} (bs/4)$ and $F_2(bs/4, bt/4)$ also both converge to the gap probability of the Bessel field on $[0,s]$.

A similar direct inspection of  the expressions in Theorem \ref{Hkernelthm} shows that 
 in the  scaling limit $b = n\beta$, $\beta >0$, $n\to \infty$ 
\be
\lim_{n\to\infty}   \frac \beta{4n(\beta+1)}   H_{01}^{(n)}\le(\frac {\zeta \beta}{4 n(1+\beta)},\frac {\xi \beta}{4 n(1+\beta)}\ri) = \le( \frac {\zeta}{\xi} \ri)^\frac a 2 K_{_B,a}(\zeta,\xi)\ , \ \ \ 
\label{219}
\ee
 while $H_{00}^{(n)}(x,y),H_{10}^{(n)}(x,y)$  tend to zero under the same rescaling\footnote{The kernel $H^{(n)}_{11}$ does {\em not} tend to zero, but it remains bounded and hence it becomes irrelevant in the correlation functions because it appears in the lower left block.}. 
This means that the eigenvalues of the first matrix near the origin ``decouple'' from those of the second matrix, and $M_1$ behaves exactly as a one-matrix Laguerre ensemble in the scaling limit near the hard-edge.
This limit \eqref{219} of the $H^{(n)}_{\mu\nu}$ kernels is equivalent to taking the limit from the Cauchy-Laguerre model to the Meijer-G process (with $x\mapsto x/n^2$) followed by  the limit \eqref{gtob}. 
Composition of the two scalings is equivalent, up to a normalization constant, to rescaling the point-field as  $x\mapsto x/n$.

To explain why the convergence to the Bessel field is intuitively clear, the reader should refer to \eqref{LagCMM}. If $b = \beta n$ scales with $n$, the probability of finding eigenvalues of $M_2$ near the origin is suppressed (see Remark \ref{remdens} and Figure \ref{FigDensComp}): its eigenvalues recede  from the origin and do not exert any longer attraction on the eigenvalues of $M_1$, which now behaves as in a one-matrix model with a hard edge and thus falls within the same universality class as the Laguerre ensemble. This intuition is based on the electrostatic interpretation of the probability density as explained in \cite{Bertola:CauchyMM}.

\br
\label{remdens}
It is explained in \cite{BertolaBalogh, Bertola:CauchyMM} that 
the  limiting (macroscopic) densities of eigenvalues\footnote{In our definition of the measure \eqref{LagCMM} the correct macroscopic scaling is to consider the eigenvalues of $\frac 1 n M_j$; had we defined the measure with ${\rm e}^{-n\tr(M_1+M_2)}$ then we would consider directly the eigenvalues of the $M_j$'s.}  of $\frac 1 n M_1,\frac 1 n  M_2$ can be computed  from the jumps $\frac {|\mu_+-\mu_-| }\pi$ on $z\in \R_+$ (for $\rho_1(z)$) and on $z\in \R_-$ (for $\rho_2(-z)$)  of the three branches of the algebraic curves below (using the first equation for $a$ and $b$ fixed, and the second for
$a$ fixed and $b=\beta n$) 
\bea
{\mu }^{3}-\frac \mu  3+{\frac {2}{27}}-\frac 1{{z}^{2}}=0 
,\qquad {\mu }^{3}-{\frac { \left( {z}^{2}+z\beta+{\beta}^{2} \right) 
\mu }{{3z}^{2}}}-
{\frac {3\,z ({\beta}^{2} +9\beta + 9) +2\,{\beta}^{3}-3\,{z}
^{2}\beta-2\,{z}^{3}}{27 {z}^{3}}}=0. 
\label{denscurves}
\eea
 One can verify that in the first case the behaviour of the densities near the origin is $z^{-\frac 23}$ while in the second case one of the densities behaves like $z^{-\frac 1 2}$ (and the other is zero), see Fig. \ref{FigDensComp}.
\er
\subsection{Outlook: computation of the gap probabilities and integrable PDEs}
Although the formulas \eqref{F1}, \eqref{F2} do compute the statistics of the lowest eigenvalues, they are transcendental and a connection with a nonlinear ordinary differential equation  (or partial DE for \eqref{F2}) is desirable. It should be pointed out, however, that the {\em numerical} computation of Fredholm determinants is not harder (in fact far simpler) than the numerical integration of nonlinear differential equations \cite{Bornemann}:  the graphs for $F_1^{(j)}(s)$ and $F_2(s,t)$ in Fig. \ref{FigFred} and \ref{FigFred2} are computed in few minutes on a low-end machine using the algorithm explained in \cite{Bornemann} and provide more than $4$  significant digits (see  Fig. \ref{FigFrederr})\footnote{A Maple worksheet to compute these determinants is available upon request.}. The main approach of Tracy and Widom \cite{TracyWidomLevel, Tracy-Widom-Bessel} is to derive Hamiltonian equations for the evaluations of the resolvent at the endpoints of the interval. 

A different  approach (which has been followed in \cite{BertolaCafasso1,BertolaCafasso2})  relies upon the theory of ``integrable kernels''  of Its-Izergin-Korepin-Slavnov (IIKS theory for short) \cite{ItsIzerginKorepinSlavnov} and it relates it to the solution of a Riemann--Hilbert problem.
These (matrix valued) kernels are of the general form \cite{ItsHarnad}
\be
K(x,y) = \frac {F^t(x)G(y)}{x-y} \ ,\ \  F,G\in Mat(p\times r)\label{intkernel}
\ee
with the property that they are nonsingular on the diagonal, namely, $F^t(x) G(x) \equiv 0$ for all $x$.
%

It is thus an important step to present the kernels in a form similar to \eqref{intkernel}. This is the purpose of the expressions in Proposition \ref{propConcomitant}.
The connection to a Riemann--Hilbert problem, once established, allows one  to derive nonlinear PDEs for the gap probabilities: this is the approach of \cite{BorodinDeift} and also \cite{BertolaCafasso1}. 
The asymptotic kernels are not of this form (except for the ``diagonal'' ones): this is not necessarily discouraging, since it was shown in \cite{BertolaCafasso2} that it is still possible to use to the IIKS theory even for kernels that are not immediately of  the form \eqref{intkernel}.

Preliminary results (in preparation with S. Y. Lee) show that the  gap $F_{2}(s,t)$ solves a particular case of  the third Painlev\'e\ transcendent in the variable $\sqrt{s+t}$, much in the same spirit as for the Tracy-Widom distribution.
It is our plan to address the connection  with Riemann--Hilbert problems and isomonodromic deformation equations in forthcoming publications.

\section{From Jacobi  to Cauchy-Laguerre biorthogonal polynomials; preliminaries}
\label{setup}
As was observed in \cite{Borodin:Biorthogonal}, the Cauchy biorthogonal polynomials in \eqref{CBOPdef}  are related to the classical Jacobi orthogonal polynomials for the weight $x^{a+b}\d x$ on $[0,1]$. We thank A. Borodin for pointing out this connection whose main point is as follows:
consider the bi-moment matrix

$$I_{ij}=
\iint_{0}^{+\infty} x^iy^j \frac{e^{-(x+y)}}{x+y}x^ay^b \, d\, x d\,y \ .$$
With the change of variables $ x=rs\ , y= r(1-s)$
 the integral becomes
 \bea
\int_0^{+\infty} r^{i+j+a+b} e^{-r}\d r\int_0^1 \frac{t^{i+a+1} (1-t)^{j+b+1}\d t}{t(1-t)}=
 \frac{\Gamma(a+i+1)\Gamma(b+j+1)}{(i+a+j+b+1)}
\ .
\eea
Notice now that the Hankel moment matrix $M$ for the Jacobi polynomials on $[0,1]$ with weight 
$x^{a+b}$ is given by 
\be
M_{ij}=\ds{\int_0^1 x^{i+j}x^{a+b} d\, x}= \frac{1}{i+j+a+b+1} \Rightarrow I_{ij}=\Gamma(a+i+1)M_{ij}\Gamma(b+j+1).  \label{MomJac}
\ee
This immediately implies  the following Proposition.
\bp
\label{propJCBOP}
Let  $P_m(x)=
\ds{\sum_{i=1}^mC_{m,i}x^i}$ denote the $m$-th Jacobi orthogonal polynomial normalized as in \eqref{Pncoeff}.\footnote{Here we define the orthogonality on $[0,1]$ rather than on $[-1,1]$.
The latter is more customary, but the correspondence between the two versions  is a simple affine transformation of the independent variable $x\to 2x-1$.} Then 
\be
\wt p_m(x):=\sum_{i=0}^m C_{m,i}\frac{x^i}{\Gamma(a+i+1)}, \qquad \wt q_n(y):=\sum_{j=1}^n
C_{n,j}\frac{y^j}{\Gamma(b+j+1)}
\label{JactoCBOPs}
\ee are the Cauchy biorthogonal 
polynomials (not orthonormal) associated with densities $
x^a {\rm e}^{-x}$ and $ y^b {\rm e}^{-y}$  and
\be
\int_{\R_+^2} \wt p_m(x) \wt q_\ell (y) \frac { x^a y^b {\rm e}^{-x-y}}{x+y} \d x \d y = 
\int_0^1 P_m(x) P_\ell (x) x^\alpha \d x = 
h_m \delta _{m\ell}
\ee
where $h_m=\frac 1{2m+\alpha + 1}$. 
  This result is valid as long as $a+1>0, b+1>0, a+b+1>0$.
\ep

\begin{remark}{\rm 
The partition function of the corresponding Cauchy two-matrix model can   be easily computed using Proposition \ref{propJCBOP} :
in \cite{Bertola:CauchyMM} it was shown that if  $c_n$ are the norms of the {\em monic} polynomials (denoted here provisionally by $p_n^{(monic)},q_n^{(monic)}$), namely, 
$
\int_{\R_+^2}  p_m^{(monic)}(x) q_\ell^{(monic)} (y) \frac { x^a y^b {\rm e}^{-x-y}}{x+y} \d x \d y = c_m \delta_{m\ell}
$, 
 then $\mathcal Z_n = (n!)^2 \prod_{j=0}^{n-1} c_j$.  Inspection of the leading coefficients of $\wt p_n, \wt q_n$ in Theorem \ref{wtpnthm} shows that
\bea
c_n&\&=  
\le[\frac {n! \Gamma(\alpha + n + 1)}{\Gamma(n+\frac {\alpha + 1}2)\Gamma(n+1+ \frac {\alpha}2)
  }  \ri]^2 \frac {\Gamma(a + n +1 )\Gamma(b + n + 1)\pi }{2^{2n+\alpha }(n + \frac{ \alpha + 1}2)}.   
\eea
Hence
\be
\mathcal Z_n = \frac {\Gamma\le(\frac {\alpha + 1}2\ri)}{\Gamma\le(n  + \frac {\alpha+1}2\ri)} 
 \frac {}{} \frac {G(n+1)^2G(\alpha+n+1)^2G(\frac{\alpha+1}2)^2 G(\frac \alpha 2 +1)^2}{G(\alpha+1)^2G(\frac {\alpha+1}2 + n )^2 G(n+1+\frac \alpha 2)^2}\frac{ G(a+n+1)G(b+n+1) \pi}{G(a+1)G(b+1) 2^{n(n-1) + n\alpha }}\ .
\ee
Here (and only in the above formula) $G(z)$ denotes Barnes' G-function, satisfying the relation $G(z+1) = \Gamma(z)G(z)\ ,\ \ G(1)=1$.
}
\end{remark}
\subsection{The kernels of the correlation functions}
The statistics of eigenvalues of $M_1,M_2$ is expressible in terms of four kernels  that can be expressed in terms of the CBOPs and auxiliary functions.  
In keeping with the notation of \cite{Bertola:CBOPs} 
 we introduce the auxiliary functions 
\be
p_n^{(1)}(z):= \int_{\R_+}\!\!\!\! \frac {  p_n(x) x^{a} {\rm e}^{-x} \d x}{
{z-x}
} \ ,\qquad 
q_n^{(1)}(w):= \int_{\R_+}\!\!\!\! \frac {  q_n(y) y^{b} {\rm e}^{-y} \d y}{
{w-y}
}.   \label{auxpn}
\ee
The four kernels to be computed are 
\bea
K^{(n)}_{00} (x,y) &\& := \sum_{j=0}^{n-1} p_j(x) q_j(y), \label{defK00}\\
K^{(n)}_{01} (x,y) &\& := 
{-}
\sum_{j=0}^{n-1} p_j(x) q^{(1)}_j(-y) = \int_{\R_+}\!\!\!\! \frac{ y'^b {\rm e}^{-y'} \d y'}{y+y'} K^{(n)}_{00}(x,y'), \label{defK01}\\
K^{(n)}_{10} (x,y) &\& :=
{-}
\sum_{j=0}^{n-1} p^{(1)}_j(-x) q_j(y) = \int_{\R_+}\!\!\!\!\frac{ x'^a {\rm e}^{-x'} \d x'}{x+x'} K^{(n)}_{00}(x',y), \label{defK10}\\
K^{(n)}_{11} (x,y)&\&  := \sum_{j=0}^{n-1} p^{(1)}_j(-x) q^{(1)}_j(-y) - \frac 1{x+y}=\\
&\& = \int_{\R_+^2}\!\!\!\!  K^{(n)}_{00}(x',y') \frac{ x'^a {\rm e}^{-x'} \d x'}{x+x'}\frac{ y'^b {\rm e}^{-y'} \d y'}{y+y'} - \frac 1{x+y}.  
\label{defK11}
\eea
Each of these kernels can be recovered in general  by solving a Riemann--Hilbert problem \cite{Bertola:CBOPs, Bertola:CauchyMM} but in the present setting this will be a direct computation.

\subsection{Facts on Jacobi Polynomials}
Since Jacobi polynomials will be instrumental to our computations, we recall their basic properties.
Jacobi polynomials are defined by the formul\ae
\begin{equation*}
P_n^{(\alpha,\beta)} (\xi) = 
\frac{\Gamma (\alpha+n+1)}{n!\,\Gamma (\alpha+\beta+n+1)}
\sum_{m=0}^n {n\choose m}
\frac{\Gamma (\alpha + \beta + n + m + 1)}{\Gamma (\alpha + m + 1)} \left(\frac{\xi-1}{2}\right)^m, 
\end{equation*}
\begin{equation*}
\int_{-1}^1 (1-\xi)^{\alpha} (1+\xi)^{\beta} 
P_m^{(\alpha,\beta)} (\xi)P_n^{(\alpha,\beta)} (\xi) \; d\xi =
\frac{2^{\alpha+\beta+1}}{2n+\alpha+\beta+1}
\frac{\Gamma(n+\alpha+1)\Gamma(n+\beta+1)}{\Gamma(n+\alpha+\beta+1)n!} \delta_{nm}.  
\end{equation*}

We shall use the variable $
z = \frac {1-\xi}2$ so that the orthogonality becomes
\be  \int_{0}^1 z^{\alpha} (1-z)^{\beta} 
P_m^{(\alpha,\beta)} \le(1-2z\ri)P_n^{(\alpha,\beta)} (1-2z) \; dz =
\frac{1}{2n+\alpha+\beta+1}
\frac{\Gamma(n+\alpha+1)\Gamma(n+\beta+1)}{\Gamma(n+\alpha+\beta+1)n!} \delta_{nm}.
\ee
The case of interest to us  is $\beta = 0,\ \alpha  = a+b$ and thus we shall introduce a simplified notation 
\begin{align}  
P_n(z)&:= P_n^{(\alpha,0)} (1-2z) = \sum_{m=0}^n 
\frac{ \Gamma (\alpha + n + m + 1)}{m! (n-m)!\Gamma (\alpha + m + 1)} \left(-z\right)^m
\label{Pncoeff}\\
&= \int_\gamma \frac {\Gamma(\alpha + n + 1-u) \Gamma(u)}{\Gamma(n+1+u) \Gamma(\alpha -u +1)} z^{-u} \frac {\d u }{2i\pi} = \frac {\Gamma(n + \alpha +1)}{n! \Gamma(\alpha +1)} {_2F_1}\le({-n,n+\alpha + 1\atop \alpha + 1};z\ri), 
\label{Pnhyper}\\
 \int_{0}^1&P_m \le(z\ri)P_n (z)  z^{\alpha} \; dz =
\frac{1}{2n+\alpha+1} \delta_{nm}\label{Pnhn}, 
\end{align} 
where the integral representation is valid for $0<|z|<1$.
In \eqref{Pnhyper} the contour encloses $u\in \R_-$ (Fig. \ref{MeijerContour}); note that the poles of $\Gamma(u)$ for $u=-n-1, -n-2,\dots$ are cancelled by the poles of the term $\Gamma(n+1+u)$ in the denominator and thus the result is a polynomial as a consequence of a simple residue computation.

\begin{wrapfigure}{r}{0.33\textwidth}
\resizebox{0.3\textwidth}{!}{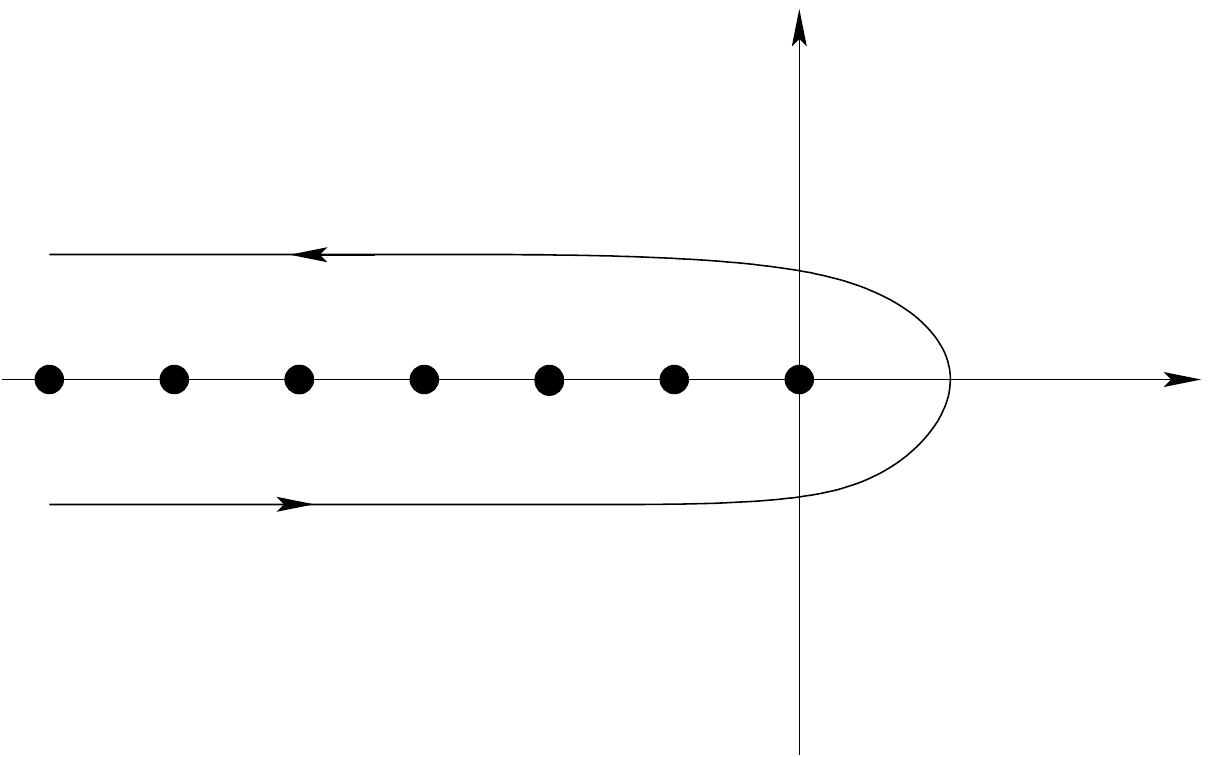}
\caption{The typical contour for Meijer-type integrals.}
\label{MeijerContour}
\end{wrapfigure}

The Christoffel-Darboux kernel for Jacobi polynomials 
\cite{Szego}
 is defined as 
\be
J_n(z,w):= \sum_{m=0}^{n-1} \frac {P_m(z) P_m(w)}{h_m},  \label{Jndef}
\ee
 where $h_n$ are the norms squared of $P_n$  in \eqref{Pnhn}. The following proposition is simple but we could not find it  in the literature, and  since its proof is very close to proofs later on we present it here.
\bp
\label{kernelJacobin}
For $\beta =0$, the Christoffel-Darboux kernel for Jacobi polynomials can be expressed as 
\bea
K_n(z,w) &\& :=\sum_{j=0}^{n-1}(2j + \alpha + 1) P_j(z)P_j(w)=\cr 
&\&= \int_{\gamma^2} \frac {\d u}{2\pi i} \frac {\d v}{2\pi i} \frac{z^{-u} w^{-v}}{1+ \alpha - v-u}  
 \frac {\Gamma(u)\Gamma(v) \Gamma(n + \alpha  + 1 - u)\Gamma( n + \alpha + 1 - v)}{\Gamma(n+u)\Gamma(n+v)\Gamma(\alpha+1 - u )\Gamma(\alpha+1 - v)}, 
 \label{211}
\eea
 valid for $0<|z|, |w|<1.$
Alternatively, introducing the function $\wh G_{n}(z)$ {\rm(a polynomial of degree $n-1$)}
\be
\wh G_n(z):= \int_\gamma \frac {\d u }{2\pi i} \frac {\Gamma(u)\Gamma(n+\alpha +1 -u)}{\Gamma(n+u)\Gamma(\alpha +1 -u)} z^{-u} 
\ee
we can write 
\be
K_n(z,w) = \int_0^1 \wh G_{n}(tz)\wh G_{n}(tw) t^\alpha \d t.  \label{213}
\ee
\ep
At this point we recall the definition of the Meijer $G$ functions, or rather a class of $G$-functions 
pertinent to this paper.  
\begin{definition} [Meijer $G$-functions]
Suppose two pairs of natural numbers $p\leq q$ and $0\leq m\leq q, \, 0\leq n\leq p$ are given.  
Then the Meijer $G$-function is defined by its Mellin-Barnes integral as follows: 
\be
G_{p,q}^{\,m,n} \!\left( \le.{ a_1, \dots, a_p \atop  b_1, \dots, b_q } \; \right| \, z \right) = \int_\gamma \frac {\d u }{2\pi i}\, \frac{\prod_{j=1}^m \Gamma(b_j + u) \prod_{j=1}^n \Gamma(1 - a_j -u)} {\prod_{j=m+1}^q \Gamma(1 - b_j -u ) \prod_{j=n+1}^p \Gamma(a_j +u )} \,z^{-u} ,
\ee
The contour $\gamma$, depicted in Fig. \ref{MeijerContour}, is chosen to encircle all the poles of functions  $\Gamma(b_j + u)$ and none of the poles of functions $\Gamma( 1 - a_j - u)$ (the implicit assumption is that none of the poles of the former coincides with any of the poles of the latter). 
\end{definition}  Further properties of $G$-functions are discussed in the Appendix \ref{MGapp}.  We can identify all the functions discussed so far in terms of $G$-functions. We have 
\begin{proposition} \label{prop:pnghatn}
The Jacobi polynomials $P_n$ and the function $\wh G_n$ appearing in the 
representation of the Christoffel-Darboux kernel are $G$-functions with symbols
\begin{equation*}
P_n(z)=G^{1,1}_{2,2} \left(\left.{-\alpha -n,n+1\atop 0,-\alpha}\; \right | z\right),\quad 
\wh G_n(z)=G^{1,1}_{2,2} \left(\left.{-\alpha -n,n\atop 0,-\alpha}\; \right | z\right).  
\end{equation*}
\end{proposition} We shall need the following Lemma, whose elementary proof is omitted.  
\bl \label{nicesum}
Let $\Omega^{j}_{k}=\frac{\Gamma( \alpha + j + 1 - u)\Gamma( \alpha + j + 1 - v)}{\Gamma(  k + u)\Gamma( k  + v)}$.  Then 
\begin{equation*}
\sum_{j=0}^{n-1}(2j + \alpha +1) \Omega^j_{j+1} =\frac{\Omega^n_n-\Omega_0^0}{1+\alpha-u-v}=
\int_0^1 t^{\alpha-u-v} (\Omega^n_n-\Omega_0^0)\d t. \end{equation*}
\el

\begin{proof} {[\bf Proof of Prop. \ref{kernelJacobin}]}
From \eqref{Jndef} and the value of the norms squared given by \eqref{Pnhn} we have
\bea
&\& \sum _{j=0}^{n-1} (2j + \alpha + 1) P_j(z) P_j (w) = \notag\\
&\& \int_{\gamma\times\gamma} \frac {\d u}{2\pi i} \frac {\d v}{2\pi i }
\frac { z^{-u} w^{-v} \Gamma( u )\Gamma(v)  } {\Gamma(\alpha  - u +1)\Gamma(\alpha  - v +1)} 
\sum_{j=0}^{n-1}\frac {(2j + \alpha +1) \Gamma( \alpha + j + 1 - u)\Gamma( \alpha + j + 1 - v)}{\Gamma(  j + 1  + u)\Gamma( j  + 1  + v)}.  \label{212}
\eea
Applying now Lemma \ref{nicesum} to \eqref{212} we obtain
\begin{align*} 
& \int_{\gamma\times\gamma} \frac {\d u}{2\pi i} \frac {\d v}{2\pi i} \frac{z^{-u} w^{-v}}{1+\alpha -u -v }  
\le[ \frac {\Gamma(u)\Gamma(v) \Gamma(n + \alpha  + 1 - u)\Gamma( n + \alpha + 1 - v)}{\Gamma(n+u)\Gamma(n+v)\Gamma(\alpha - u + 1)\Gamma(\alpha - v + 1)}
-1\ri]=\\&\int_{\gamma\times\gamma} \frac {\d u}{2\pi i} \frac {\d v}{2\pi i}\int_0^1  t^{\alpha}  (tz)^{-u} (tw)^{-v}
\le[ \frac {\Gamma(u)\Gamma(v) \Gamma(n + \alpha  + 1 - u)\Gamma( n + \alpha + 1 - v)}{\Gamma(n+u)\Gamma(n+v)\Gamma(\alpha - u + 1)\Gamma(\alpha - v + 1)}
-1 \ri]\! \d t .  
\end{align*}
The term with the $1$ vanishes identically  because the contours of integration can be retracted to $-\infty$ and $z, w\in (0,1)$, 
 and the denominator does not vanish ($\alpha >-1$) because the contour around $0$ can be moved sufficiently to the left so that the denominator has no zero anywhere inside the contour. 
This immediately yields the first formula \eqref{211}. To obtain \eqref{213} we switch the 
order of integration in the second integral formula above.  
\end{proof}

Using the Stirling approximation formula one obtains readily the following Lemma.
\bl
\label{lemmagamma}
The following asymptotic formula holds uniformly in each sector  $|\arg(z)| < \pi-\epsilon$ ;
\bea
\frac {\Gamma ( z + \delta)}{\Gamma(z + \rho)} = z^{\delta - \rho} 
\le(
1 + \frac {(\delta + \rho -1) (\delta - \rho)}{2z} + \mathcal O(z^{-2}) 
\ri) =z^{\delta - \rho}
\le(
\frac {z+1}z\ri)^{\frac {(\delta + \rho -1) (\delta - \rho)} 2}\hspace{-40pt} \le(1 + \mathcal O(z^{-2}) 
\ri) ,\ \ |z|\to\infty.
\eea
\el
Then, using Lemma \ref{lemmagamma} and Proposition \ref{kernelJacobin}
\bea
&\& K_n\le(\frac x{n^2} \le(\frac {n}{n+1}\ri)^\alpha, \frac {y} {n^2} \le(\frac {n}{n+1}\ri)^\alpha \ri) =\notag \\
&\& =n^{2\alpha +2} \le(\frac {n+1}n\ri) ^ {\alpha(\alpha+1)} \int_{\gamma^2} \frac {\d u}{2\pi i} \frac {\d v}{2\pi i} \frac{x^{-u} y^{-v}}{1+ \alpha - v-u}  
 \frac {\Gamma(u)\Gamma(v)}{\Gamma(\alpha - u + 1)\Gamma(\alpha - v + 1)} \le(1+ \mathcal  O(n^{-2})\ri)=\cr
 &\& =n^{2\alpha +2}\le(\frac {n+1}n\ri) ^ {\alpha(\alpha+1)} \le(1  + \mathcal  O(n^{-2})\ri) \int_0^1 t^\alpha  B_\alpha( t x) B_\alpha(ty) \d t,  \label{KernelJac}
 \eea
 where we have introduced the function $B_\alpha$  defined as 
 \bea
&\&  B_\alpha(z):=  \int_{\gamma} \frac {\d u}{2\pi i} z^{-u} 
 \frac {\Gamma(u)}{\Gamma(\alpha - u + 1)} = z^{-\alpha/2} J_\alpha(2\sqrt{z}).   
\eea
In the above $J_\alpha$ denotes the Bessel function of the first kind.  We also note that the integral makes sense for any value of $z$ thanks to the super-exponential behavior of the gamma functions.  Finally, $B_{\alpha}$ itself is a Meijer G-function, namely 
$G^{1,0}_{0,2} \left(\left.{~\atop 0,-\alpha}\; \right | z\right)$.  
The leading term of the kernel in \eqref{KernelJac} is almost the same as the Bessel kernel. Indeed, using  eq. (2.2)  in \cite{Tracy-Widom-Bessel} we can write the Bessel kernel, denoted here by $K_{_B,\alpha} $, as
\bea
K_{_B,\alpha }(x,y) &\& := \frac {J_\alpha (\sqrt{x})\sqrt{y} J'_\alpha(\sqrt{y}) -J_\alpha (\sqrt{y})\sqrt{x} J'_\alpha(\sqrt{x}) }{2(x-y)}=  \frac 1 4 \int_0^1 J_\alpha(\sqrt{xt}) J_{\alpha} (\sqrt{yt}) \d t.  
\label{BesselKernel0}
\eea 
A direct inspection shows that 
\bea
K_{_B,\alpha }(4x,4y) =  \frac 1 4 \int_0^1 J_\alpha(2\sqrt{xt}) J_{\alpha} (2\sqrt{yt}) \d t =  (xy)^{\frac \alpha 2}\frac 1 4 \int_0^1t^\alpha B_\alpha(tx) B_\alpha(ty)  \d t. 
\label{BesselKernel}  
\eea
Thus 
$$K_n\le(\frac x{n^2} \le(\frac {n}{n+1}\ri)^\alpha, \frac {y} {n^2} \le(\frac {n}{n+1}\ri)^\alpha \ri)     =n^{2\alpha +2}\le(\frac {n+1}n\ri) ^ {\alpha(\alpha+1)} \le(1 + \mathcal  O(n^{-2})\ri)\frac{4K_{B}(4x,4y)}{(xy)^{\frac{\alpha}{2}}}$$
This is not surprising since the Bessel kernel arises precisely as the scaling limit of the Jacobi (and Laguerre) ensemble \cite{Tracy-Widom-Bessel}.
 \br
The various factors $\frac {n+1}n$ in the limit of large $n$ are largely irrelevant, since they contribute to order $\mathcal O(n^{-1})$. Their introduction above yields a $\mathcal O(n^{-2})$ error term, as Lemma \ref{lemmagamma} shows.  
\er

\section{Cauchy-Laguerre Biorthogonal polynomials and their kernels}
\label{CLBOPS}
The reason for the use of  Laguerre's  name is clearly justified by the shape of the weights; nevertheless, as we are going to see, the resulting polynomials are more closely related to Jacobi polynomials. 
 Using the explicit form of Jacobi polynomials as given by \eqref{Pncoeff}, their relation to Cauchy BOPs (see \eqref{JactoCBOPs}) and the well known fact $\res{u=-k} \Gamma(u)= (-1)^k \frac 1{k!}$ we can immediately assert the following
\bt
\label{wtpnthm}
The  polynomials $\wt p_n$ of Proposition \ref{propJCBOP} are given by 
\bea
\wt p_n(z) &\& = \sum_{j=0}^n \frac {\Gamma(\alpha + n + j + 1)(-z)^j}{j! (n-j)! \Gamma(\alpha + j + 1)\Gamma(a + j + 1)} = 
\frac {\Gamma(\alpha + n + 1)}{n! \Gamma(\alpha + 1)\Gamma(a + 1)}
{_2F_2}\le({-n,\alpha + n + 1\atop  a+1, \alpha + 1}; z\ri) =
\nonumber \\
&\& =\int_\gamma \frac {\d u }{2i\pi} \frac {\Gamma( \alpha + n - u + 1)\Gamma(u)}{\Gamma( n +1 + u)\Gamma(\alpha  +1 - u)\Gamma (a + 1 - u)} z^{-u}
\label{wtpnMeij}
\eea
where the contour $\gamma$ is the same as in Fig. \ref{MeijerContour} with the poles of $\Gamma(\alpha + n +1 -u)$ in its exterior\footnote{The notation for generalized hypergeometric functions follows that of \href{http://dlmf.nist.gov/16.2}{16.2 in DLMF}.}.  
The integral representation is valid for any $z\neq 0$.
The expressions for the polynomials $\wt q_n$ follow from those for $\wt p_n$ by exchanging $a\leftrightarrow b$.
\et

We have two simple corollaries
\bc[Averages of $\wt p_n, \wt q_n$]\label{cor:avtpnqn}
The average of $\wt p_n$ with respect to the measure $x^a {\rm e}^{-x}\d x$, \rm($\wt q_n$ with 
respect to $y^b {\rm e}^{-y}\d y$\rm) is 
\be
\int_0^\infty x^a {\rm e}^{-x} \wt p_n(x) \d x = (-1)^n = \int_0^\infty y^a {\rm e}^{-y} \wt q_n(y) \d y. 
\ee
\ec
{\bf Proof.}
A simple computation gives: 
\begin{equation*} 
\int_0^\infty x^a {\rm e}^{-x} \wt p_n(x) \d x  = \sum_{j=0}^{n} \frac {\Gamma(\alpha + n + j + 1)(-1)^j \int_0^\infty x^{j+a} {\rm e}^{-x} \d x}{j!(n-j)!\Gamma(\alpha + j + 1) \Gamma( a + j + 1)} = \sum_{j=0}^{n} \frac {\Gamma(\alpha + n + j + 1)(-1)^j}{j!(n-j)!\Gamma(\alpha + j + 1) } =P_n(1). 
\end{equation*} 
It is well known that $P_n(1) = (-1)^n$ as a consequence of
\eqref{Pncoeff}.
\QED
The second corollary is an elementary identification of $\wt p_n$, $\wt q_n$ with one of the Meijer $G$-functions.    
\bc The polynomials $\wt p_n$ and $\wt q_n$ can be identified with the following 
Meijer $G$-functions: 
$$
\wt p_n(z)=G^{1,1}_{2,3} \left(\left.{-\alpha -n,n+1\atop 0,-\alpha, -a}\; \right | z\right), \qquad 
\wt q_n(z)=G^{1,1}_{2,3} \left(\left.{-\alpha -n,n+1\atop 0,-\alpha, -b}\; \right | z\right). 
$$
\ec
Direct inspection of the leading coefficient of $\wt p_n, \wt q_n$ together with 
Corollary \ref{cor:avtpnqn} imply the next result
\bt
\label{thmpnqn}
Let $\pi_n := \sqrt{2n + \alpha +1}\sqrt{\frac {\Gamma(a + n +1)}{\Gamma(b + n + 1)}}
, \, \quad \eta_n:= \sqrt{2n + \alpha +1}\sqrt{\frac {\Gamma(b + n + 1)}{\Gamma(a + n +1)}}$.  Then the {\bf bi-orthonormal polynomials} $p_n,q_n$, normalized to have identical positive leading coefficients, are given by
$$
 p_n(z) 
=(-1)^n\pi _n  \wt p_n, \quad 
q_n(z)  =(-1)^n \eta_n \wt q_n. $$
Moreover, $\pi_n = \int_0^\infty x^a {\rm e}^{-x} p_n(x) \d x  $ and 
$\eta_n= \int_0^\infty y^b {\rm e}^{-y} q_n(y)\d y.$ 
\et
\subsection{The auxiliary functions $p_n^{(1)}$ and $p_n^{(2)}$}
Since the definition of $q_n^{(1)}$can be obtained from that of $p_n^{(1)}$ by swapping $a$ with $b$ we subsequently will focus only on $p_n^{(1)}$ as defined in \eqref{auxpn}.\newline
Our goal is to express $p_n^{(1)}(z)$ in terms of the $G$-functions.  To this end we first 
compute $\int_0^\infty \frac{x^a {\rm e}^{-x} \wt p_n(x)}{z+x} \d x$ for $z>0$.  
\begin{lemma} \label{lem:biglemma}
Let $z>0, a>-1, b>-1, a+b=\alpha>-1$.  Then 
\begin{equation*}
\int_0^\infty \frac{x^a {\rm e}^{-x}\wt p_n(x)}{z+x} \d x={\rm e}^z G^{2,1}_{2,3} \left(\left.{-n-b,a +n+1\atop 0,a, -b}\; \right | z\right).  
\end{equation*}
\end{lemma} 
\begin{proof} 
By Theorem \ref{wtpnthm} 
\begin{equation*}
\int_0^\infty \frac{x^a {\rm e}^{-x}\wt p_n(x)}{z+x} \d x=\int_\gamma \frac {\d u }{2\pi i} \frac {\Gamma( \alpha + n - u + 1)\Gamma(u)}{\Gamma( n +1 + u)\Gamma(\alpha  +1 - u)\Gamma (a + 1 - u)}\int_0^\infty  \frac{x^{a-u}{\rm e}^{-x}}{z+x} \d x,   
\end{equation*}
where the contour $\gamma$ is chosen for this computation in such a way that $\Re u<1+a$ 
in addition to $\Re u<1+\alpha$ required by Theorem \ref{wtpnthm}.  Furthermore, 
we move the contour $\gamma$ sufficiently to the right to ensure that $u=a$ is inside the contour.  The change of the order of integration is justified because both single variable integrals are uniformly convergent on compact sets of their respective complementary 
variables and the original iterated integral is 
absolutely convergent.  
The inner integral is up to a normalization the Kummer function of the second kind (see 
\cite{BW} p. 194).  
A quick way of computing this integral is to substitute $\frac{1}{z+x}=\int_0^{\infty} {\rm e}^{-(z+x)s} \d s$ then switch the order of iterated integrals to get: 
$$
\int_0^\infty  \frac{x^{a-u}{\rm e}^{-x}}{z+x} \d x={\rm e}^z \Gamma(a+1-u)z^{a-u}\Gamma(u-a;z), $$
where $\Gamma(u-a;z)$ is the incomplete gamma function.  To conclude this part of the computation 
one only needs to use $\Gamma(u-a;z)=\Gamma(u-a)+\gamma(u-a;z)=\Gamma(u-a)-
\frac{z^{u-a}}{a-u}~_1F_1(u-a,u-a+1;-z)$ which expresses the complementary incomplete 
gamma function $\gamma$ in terms of the hypergeometric function $~_1F_1$ (see 
\cite{Bateman1} p. 266, formula (22)) thus obtaining
$$
\int_0^\infty  \frac{x^{a-u}{\rm e}^{-x}}{z+x} \d x={\rm e}^z \Gamma(a+1-u)\Gamma(u-a)z^{a-u}-{\rm e}^z \Gamma(a-u)~_1F_1(u-a,u-a+1;-z).  
$$
Putting the first term on the right side back into the contour integral 
gives the claim if one uses the shift formula \eqref{eq:shifting}.  
To finish the proof we need to show that the contour integral of the second term 
vanishes.  To this end we write explicitly the resulting contour integral obtaining, after omitting ${\rm e}^z$ and performing elementary computations with the gamma functions, 
\begin{equation*}
\int_\gamma \frac {\d u }{2\pi i} \frac {\Gamma( \alpha + n - u + 1)\Gamma(u)}{\Gamma( n +1 + u)\Gamma(\alpha  +1 - u)}\sum_{j=0}\frac{1}{(u-a+j)j!} (-z)^j. 
\end{equation*}
 The series converges uniformly on $\gamma$ hence to prove that this integral is zero it suffices to prove 
 \begin{equation*}
\int_\gamma \frac {\d u }{2\pi i} \frac {\Gamma( \alpha + n - u + 1)\Gamma(u)}{\Gamma( n +1 + u)\Gamma(\alpha  +1 - u)}\frac{1}{(u-a+j)} =0, \quad j=0,1,\dots
\end{equation*}
The integrand is a meromorphic function with poles at $u=\{0,-1, \cdots, -n\} \cup \{a-j\}$ which, 
according to our choice of the contour $\gamma$, are all inside the contour.  
We can now extend this contour by adding a large circle with the center at $u=0$ and 
radius $r$ and observe that the integrand on such a large circle is, in view Lemma \ref{lemmagamma},  $\mathcal O(u^{-2})$.  This implies that in the limit of $r\rightarrow \infty$ 
the integral $\int_\gamma \frac {\d u }{2\pi i} \frac {\Gamma( \alpha + n - u + 1)\Gamma(u)}{\Gamma( n +1 + u)\Gamma(\alpha  +1 - u)}\frac{1}{(u-a+j)}=[\text{ sum of residues of the 
integrand outside of $\gamma$}]=0$.  
\end{proof} 
The formula for $p_n^{(1)}(z)$ follows now easily from Lemma \ref{lem:biglemma}  and 
Theorem \ref{thmpnqn}.  
\bt \label{auxpnthm}
For $z\in \C\setminus \R_+$ the auxiliary functions of the first kind $p_n^{(1)}$  have the following Meijer $G$-function representation
\begin{equation*}
p_n^{(1)}(z):=  \int_0^\infty \frac {p_n(x) x^a {\rm e}^{-x}\d x}{z-x}={\rm e}^{-z} (-1)^{n+1}\pi_n G^{2,1}_{2,3} \left(\left.{-n-b,a +n+1\atop 0,a, -b}\; \right |- z\right).  
\end{equation*}
A similar expression holds for $q^{(1)}_n$ by interchanging $a\leftrightarrow b$ and $\pi_n \leftrightarrow \eta_n$.\et
To see the relation between $p_n^{(1)}$ and $p_n$ we formulate the following equivalent 
representation of $p_n^{(1)}$ whose gamma part is identical to 
that for $p_n$; see \eqref{wtpnMeij}.  
\bc
For $z\in \C\setminus \R_+$ the auxiliary functions of the first kind admit an equivalent 
integral representation: 
\bea
p_n^{(1)}(z) =  {\rm e}^{-z} \ (-1)^{n+1}
  \pi_n \int_{ \gamma} \!\!\frac {\d u}{2\pi i} \frac {\Gamma(1+n+\alpha -u)\Gamma(u)}{\Gamma(
  1+\alpha -u)\Gamma(n+1+u)\Gamma(1+a-u)}(- z)^{a-u}\frac{\pi}{\sin(\pi(u-a))}
\eea

\ec
\begin{proof} 
 The $G$-function occurring in the representation of $p_n$ is $G^{1,1}_{2,3} \left(\left.{-n+\alpha, n+1\atop 0,-\alpha, -a}\; \right | z\right)=G^{1,1}_{2,3} \left(\left.{-n+\alpha, n+1\atop 0,-a, -\alpha}\; \right | z\right)$.  Likewise, in view of Theorem \ref{auxpnthm} and the shift formula \eqref{eq:shifting} 
the $G$-function in the representation of $p_n^{(1)}$  can be written as 
$(-z)^aG^{2,1}_{2,3} \left(\left.{-n-\alpha, n+1\atop -a,0, -\alpha}\; \right | -z\right)$.  
  Switching the first and the second coefficient in the bottom line 
of the latter $G$ symbol can be done if one uses {\sl Euler's reflection formula} to write 
$\Gamma(u-a)=\frac{ \pi }{\Gamma(1+a-u) \sin (\pi(u-a))}$ which implies the claim.
\end{proof} 
%
\begin{theorem} For $z\in \C\setminus \R_-$ the auxiliary functions of the second kind admit the $G$-function representation
\begin{equation}
p_n^{(2)}(z):=  -\int_0^\infty \int_0^\infty \frac {p_n(x) x^a y^b {\rm e}^{-x-y}\d x \d y }{(x+y)(z+y)}=(-1)^{n+1} \pi _n G^{3,1}_{2,3} \left(\left.{-n,\alpha +n+1\atop 0,\alpha, b}\; \right | z\right)
\end{equation}
\end{theorem} 
\begin{proof} Let us express $p_n^{(2)}$ in terms of $p_n^{(1)}$ as 
\begin{equation*}
p_n^{(2)}(z)=\int_0^{\infty} \frac{y^b e^{-y}}{z+y} p_n^{(1)}(-y) \d y.  
\end{equation*}
Since $p_n^{(1)}(-y)=(-1)^{n+1} \pi_n \, {\rm e} ^{y}\, G^{2,1}_{2,3} \left(\left.{-n-b,a+n+1\atop a,0, -b}\; \right | y\right)$ we obtain 
\begin{equation*}
p_n^{(2)}(z)=(-1)^{n+1} \pi_n \,\int_0^{\infty} \frac{y^b }{z+y}G^{2,1}_{2,3} \left(\left.{-n-b,a+n+1\atop a,0, -b}\; \right | y\right)\d y=(-1)^{n+1} \pi_n \,G^{3,1}_{2,3} \left(\left.{-n,\alpha +n+1\atop 0,\alpha, b}\; \right | z\right), 
\end{equation*}
by \eqref{eq:Gintegral} and \eqref{eq:shifting}.  
 \end{proof}
Again to see the relation between $p_n^{(1)}$ and $p_n^{(2)}$ we establish an equivalent 
representation
\bc For $z\in \C\setminus \R_-$ the auxiliary functions of the second kind admit an 
equivalent integral representation:
\begin{equation*}
p_n^{(2)}(z)=(-1)^{n+1} \pi _n \int_{ \gamma} \!\frac {\d u}{2\pi i} \, \frac {\Gamma(u)\Gamma(u+a)\Gamma(1+n+b-u)}{\Gamma(1+b-u)\Gamma(a+n+1+u)}z^{b-u}\frac{\pi}{\sin(\pi(u-b))}. 
\end{equation*}
\ec
\begin{proof} 
Using the shift formula given by \eqref{eq:shifting} we write $G^{3,1}_{2,3} \left(\left.{-n,\alpha +n+1\atop 0,\alpha, b}\; \right | z\right)=z^bG^{3,1}_{2,3} \left(\left.{-n-b,a+n+1\atop -b,a, 0}\; \right | z\right)$. Also, in view of Theorem \ref{auxpnthm}, the $G$ function in the representation of $p_n^{(1)}$ is $G^{2,1}_{2,3} \left(\left.{-n-b,a+n+1\atop 0,a, -b}\; \right | -z\right) $.  Switching the first and the last coefficient in the bottom line 
of the $G$ symbol can be done if one uses {\sl Euler's reflection formula} to write 
$\Gamma(u-b)=\frac{ \pi }{\Gamma(1+b-u) \sin (\pi(b-u))}$ which implies the claim.  
\end{proof}
\br
We can easily verify the jump of $p_n^{(2)}(z)$ 
for $z<0$ to be \newline 
$p_n^{(2)} (z)_+ - p_n^{(2)}(z)_- =2\pi i
  (-z)^{b} {\rm e}^{z}p_n^{(1)}(z)$. 
Indeed,  let $z<0$ then  
$$(z)^{b-u}_+-(z)^{b-u}_-=(-z)^{b-u} \left[e^{i\pi(b-u)}-e^{-i\pi (b-u)} \right]=(-z)^{b-u}
2i\sin(\pi (b-u)), $$  
since the cut for $z^x$ is by convention (see \ref{MGapp}) along the negative real axis of $z$.   
Thus, by the corollary above
we get: 
\begin{equation*}
p_n^{(2)} (z)_+ - p_n^{(2)}(z)_-= (-z)^b (-1)^{n+1}\pi_n
G^{2,1}_{2,3} \left(\left.{-n-b,a+n+1\atop 0,a, -b}\; \right | -z\right), 
 \end{equation*}
which by Theorem \ref{auxpnthm} implies $p_n^{(2)} (z)_+ - p_n^{(2)}(z)_-=2\pi i (-z)^b{\rm e}^{z} p_n^{(1)}(z)$.  
\er

\subsection{The kernels for  finite $n$}
In \cite{Bertola:CBOPs} we proved a variety of generalizations of Christoffel--Darboux identities.  In this section, however, we show that in the case at hand it is possible to compute the associated kernels directly and in an elementary way from the formul\ae\ above.  We will begin with $K^{(n)}_{00}$ as defined in \eqref{defK00}.  
\bt
\label{K00thm}
The principal kernel $K^{(n)}_{00} (x,y)$ is given by 
\begin{align}
&K^{(n)}_{00}(x,y)  =
\int_0^1 t^\alpha G_{a,n}(tx) G_{b,n}(ty)\d t,  \quad \text{ where } 
\label{K00conv}\\
&G_{c,n}(x) := 
\int_{\gamma} \frac {\d u }{2i\pi}  \frac{ x^{-u}}
{\Gamma (c + 1 - u)
} \frac {\Gamma(u) \Gamma(n + \alpha  + 1 - u)}{\Gamma(n+u)\Gamma(\alpha - u + 1)}
=G^{1,1}_{2,3} \left(\left.{-\alpha -n,n\atop 0,-c, -\alpha}\; \right | x\right).  \label{G}
\end{align}
\et
\begin{proof}
We start with writing the formula for the kernel in terms of $G$-functions.  By 
Theorem \ref{thmpnqn} $K^{(n)}_{00} (x,y) =\sum_{j=0}^{n-1} (2j+\alpha +1)G^{1,1}_{2,3} \left(\left.{-\alpha -j,j+1\atop 0,-\alpha, -a}\; \right | x\right)G^{1,1}_{2,3} \left(\left.{-\alpha -j,j+1\atop 0,-\alpha, -b}\; \right | y\right)$, while Christoffel-Darboux kernel for Jacobi polynomials 
is given by $K_n(x,y)=\sum_{j=0}^{n-1} (2j+\alpha+1)G^{1,1}_{2,2} \left(\left.{-\alpha -j,j+1\atop 0,-\alpha}\; \right | x\right)G^{1,1}_{2,2} \left(\left.{-\alpha -j,j+1\atop 0,-\alpha}\; \right | y\right)$ (see Propositions \ref{kernelJacobin} and \ref{prop:pnghatn}).  This means that 
the integral representations obtained for $K_n(x,y)$ are only modified by 
extra factors $\frac{1}{\Gamma(1+a-u)}, \frac{1}{\Gamma(1+b-v)}$ coming from the 
additional third index appearing in the bottom of the respective $G$-symbols of $\wt p_j(x), \wt q_n(y)$ which proves the claim.  
\end{proof}
\br  
We can write explicitly the double integral representation of the kernel by copying 
the analogous formula from Theorem \ref{kernelJacobin} and attaching 
additional factors $\frac{1}{\Gamma(1+a-u)}, \frac{1}{\Gamma(1+b-v)}$ as explained in the proof above.  The formula reads: 
\begin{multline}
K^{(n)}_{00}(x,y) =\\\int_{\gamma^2} \frac {\d u }{2\pi i} \frac {\d v }{2\pi i} \frac{ x^{-u} y^{-v}}
{\Gamma (a + 1 - u)\Gamma (b+ 1 - v)(1+\alpha -v-u)
} \frac {\Gamma(u)\Gamma(v) \Gamma(n + \alpha  + 1 - u)\Gamma( n + \alpha + 1 - v)}{\Gamma(n+u)\Gamma(n+v)\Gamma(\alpha - u + 1)\Gamma(\alpha - v + 1)}.  \label{K00int}
\end{multline}
Note that this integral representation, as opposed to the one appearing in \ref{kernelJacobin}, is 
valid for any $x,y \neq 0$. Observe also that the function 
$G_{c,n} (x)=G^{1,1}_{2,3} \left(\left.{-\alpha -n,n\atop 0,-c, -\alpha}\; \right | x\right)$ is a natural generalization of the function $\wh G_n(x)=G^{1,1}_{2,2} \left(\left.{-\alpha -n,n\atop 0,-\alpha}\; \right | x \right)$ known from Theorem \ref{kernelJacobin}.  
\er
%

We are now ready to compute the remaining kernels. Using Lemma \ref{nicesum}  we find 

\bt
\label{K10thm}
The kernel $K^{(n)}_{10}(x,y)$ is given by 
 \bea
 &\& K^{(n)}_{10}(x,y)
 =x^a {\rm e}^{x} \int_0^1 t^\alpha  \wt G_{a,n}(tx)G_{b,n} (ty)\, \d t, \quad \text{where} \\
 && \wt G_{c,n}(x):=  \int_{ \gamma} \!\!\frac {\d u}{2\pi i} 
 \frac { \Gamma(u-c)\Gamma(u)\Gamma  \left( n+\alpha+1-u \right) }{ \Gamma(\alpha  + 1-u)  \Gamma  \left( n+u \right)} x^{-u}=G^{2,1}_{2,3} \left(\left.{-\alpha -n,n\atop 0,-c,-\alpha}\; \right | x\right).  \label{wtG}
\eea
Likewise, by symmetry, the kernel $K^{(n)}_{01}(x,y)$ is given by:
\bea
K^{(n)}_{01}(x,y) =y^b {\rm e}^{y}   \int_0^1 t^\alpha G_{a,n}(tx)\wt G_{b,n} (ty)\, \d t, 
 \eea
 with $\wt G_{c,n}$ defined in \eqref{wtG} and $G_{c,n}$ defined in \eqref{G}. 
\et
\begin{proof} From the definition of $K^{(n)}_{10}$ in \eqref{defK10}, Theorem \ref{auxpnthm} 
and \eqref{eq:shifting} we obtain 
\begin{align*}
&K^{(n)}_{10}(x,y) :=-\sum_{j=0}^{n-1} p^{(1)}_j(-x)q_j(y)=
x^{a}{\rm e}^{x}  
\sum_{j=0}^{n-1} (2j +\alpha+1)G^{2,1}_{2,3} \left(\left.{-\alpha -j,j+1\atop 0,-a, -\alpha}\; \right | x\right)G^{1,1}_{2,3} \left(\left.{-\alpha -j,j+1\atop 0,-\alpha, -b}\; \right | y\right).  
\end{align*}
We observe that the dependence on $j$ is identical to 
the one covered in Lemma \ref{nicesum}.  Hence by that lemma 
\begin{align*} 
&K^{(n)}_{10}(x,y)=x^a {\rm e}^{x}\le \{ \int_0^1 t^{\alpha} G^{2,1}_{2,3} \left(\left.{-\alpha-n,n\atop 0,-a, -\alpha}\; \right | tx\right)G^{1,1}_{2,3} \left(\left.{-\alpha -n,n\atop 0,-\alpha, -b}\; \right | ty\right)\! \d t-\ri.\\
&\le. \int_0^1 t^{\alpha} G^{1,0}_{0,1} \left(\left.{\atop -a}\; \right | tx\right)G^{0,0}_{0, 1} \left(\left.{\atop -b}\; \right | ty\right) \! \d t \ri\}=x^a {\rm e}^{x}\int_0^1 t^{\alpha} G^{2,1}_{2,3} \left(\left.{-\alpha-n,n\atop 0,-a, -\alpha}\; \right | tx\right)G^{1,1}_{2,3} \left(\left.{-\alpha -n,n\atop 0,-\alpha, -b}\; \right | ty\right)\! \d t, 
\end{align*}
since the second term,  or more precisely $G^{0,0}_{0, 1} \left(\left.{\atop -b}\; \right | ty\right)$, vanishes by \eqref{eq:Gzero}.  
\end{proof}
\bt
\label{K11thm}
The Kernel $K^{(n)}_{11}(x,y)$ is given by
 \be
 K^{(n)}_{11}(x,y) =x^a y^b {\rm e}^{x+y} \int_0^1  \wt G_{a,n}(tx) \wt G_{b,n}(ty) t^\alpha \d t - \frac{{\rm e}^{x+y}}{x+y}
 \label{K11conv}
 \ee
 with $\wt G_{c,n}$ defined in \eqref{wtG}.
\et
\begin{proof}
We plug the expressions for $p_n^{(1)}, q_n^{(1)}$ of Theorem \ref{auxpnthm} into the definition of $K^{(n)}_{11}$ \eqref{defK11} to get
\begin{align*}
& K^{(n)}_{11}(x,y)+\frac 1{x+y} : = \sum_{j=0}^{n-1} p^{(1)}_j(-x) q^{(1)}_j(-y) =\\
&x^a y^b {\rm e}^{x+y} \sum_{j=0}^{n-1} (2j+\alpha + 1)G^{2,1}_{2,3} \left(\left.{-\alpha -j,j+1\atop 0,-a, -\alpha}\; \right | x\right)G^{2,1}_{2,3} \left(\left.{-\alpha -j,j+1\atop 0,-b, -\alpha}\; \right | y\right). 
\end{align*}
Again, the $j$-dependence of this expression is covered by Lemma \ref{nicesum} and applying it  we obtain: 
\begin{align*}
& K^{(n)}_{11}(x,y)+\frac 1{x+y} =x^ay^b  {\rm e}^{x+y}\big \{ \int_0^1 t^{\alpha} G^{2,1}_{2,3} \left(\left.{-\alpha-n,n\atop 0,-a, -\alpha}\; \right | tx\right)G^{2,1}_{2,3} \left(\left.{-\alpha -n,n\atop 0,-b, -\alpha}\; \right | ty\right)\! \d t-\\
&\int_0^1 t^{\alpha} G^{1,0}_{0,1} \left(\left.{\atop -a}\; \right | tx\right)G^{1,0}_{0, 1} \left(\left.{\atop -b}\; \right | ty\right) \! \d t \big \}.  
\end{align*}
By inspection $G^{1,0}_{0,1} \left(\left.{\atop -a}\; \right | tx\right)=(tx)^{-a} {\rm e} ^{-tx}$, hence 
an elementary computation gives the final answer
\begin{align*}
& K^{(n)}_{11}(x,y)+\frac 1{x+y}  =x^ay^b  {\rm e}^{x+y}\le \{ \int_0^1 t^{\alpha} G^{2,1}_{2,3} \left(\left.{-\alpha-n,n\atop 0,-a, -\alpha}\; \right | tx\right)G^{2,1}_{2,3} \left(\left.{-\alpha -n,n\atop 0,-b, -\alpha}\; \right | ty\right)\! \d t+\ri.\\
&\le. \frac{1}{x^a y^b} \frac{({\rm e}^{-(x+y)}-1)}{x+y} \ri\}, 
\end{align*}
which implies the claim.  \end{proof}
\br Similar to the integral formula for the kernel $K^{(n)}_{00}$ (see \eqref{K00int}) 
all the remaining kernels have analogous integral representations in terms of the 
double path integral, all involving the kernel $\frac{1}{1+\alpha-u -v}$.  
\er
\subsection{The Meijer-G random point field}
\label{MGrpf}
In order to describe the statistics of the Cauchy matrix model near $x=0=y$   we recall the definition of the kernels for the correlation functions, that is, the $K$ and $H$ kernels defined 
by \eqref{eq:Ks}, \eqref{eq:Hs} respectively.
 
\subsubsection{Asymptotics of the kernels $K^{(n)}_{\mu\nu}$}
From the explicit integral  expressions for the functions $G_{c,n}, \wt G_{c,n}$
and the kernels in Theorems \ref{K00thm}, \ref{K10thm}, \ref{K11thm},  we shall now derive the behaviour under the rescaling $
x = \frac {\zeta} {n^2} \ ,  y = \frac { \xi} {n^2},\ \ n\to\infty.$
\br
\label{remrescaling}
It should be mentioned here that if we had started with a model in the form 
\be
\d \mu(M_1,M_2) = \frac {1 }{\mathcal Z_n}  \frac{ (\det M_1)^a(\det M_2)^b {\rm e}^{-n \tr (M_1 + M_2)}}{\det (M_1 + M_2)^n}\, \d M_1 \d M_2, 
\ee
then the relevant rescaling would have  been $x\mapsto xn^{-3}$,
because the $n$ scaling in the exponential is simply absorbed by a rescaling $\wt x=n x$, eventually giving the scaling by $n^{-2}$.  
\er

\bt[Kernels in the asymptotic regime]
\label{kerasymptthm}
Let 
\be
x = \frac \zeta {n^2}\le(\frac {n}{n+1}\ri)^\alpha ,\qquad y=  \frac \xi {n^2}\le(\frac {n}{n+1}\ri)^\alpha.  
\label{rescaling}
\ee
The following asymptotic estimates for $n\to \infty$ hold uniformly on compact subsets of the independent variables 
\bea
n^{-\alpha -1} \le(\frac {n}{n+1}\ri)^{\frac {\alpha(\alpha+1)}2} 
G_{c,n}\le (
x
\ri) &\& = \overbrace{\int_{\gamma} \frac {\d u }{2\pi i} \frac {\Gamma(u)}{\Gamma (c + 1 - u)\Gamma(\alpha +1- u )}\zeta^{-u}}^{=: G_c(\zeta )} + \mathcal O(n^{-2}), \\
n^{-\alpha -1}\le(\frac {n}{n+1}\ri)^{\frac {\alpha(\alpha+1)}2} 
\wt G_{c,n}\le (
x
\ri)
 &\& =\underbrace{  \int_{ \gamma} \!\!\frac {\d u}{2\pi i} 
 \frac { \Gamma(u-c)\Gamma(u)}{ \Gamma(\alpha  + 1-u)} \zeta ^{-u}}_{=: \wt G_c(\zeta )} + \mathcal O(n^{-2}), 
\eea
or in terms of $G$-functions 
\begin{align*} 
n^{-\alpha -1}\le(\frac {n}{n+1}\ri)^{\frac {\alpha(\alpha+1)}2}  
G^{1,1}_{2,3} \left(\left.{-\alpha -n,n\atop 0,-c, -\alpha}\; \right | 
x
 \right)
=G^{1,0}_{0,3} \left(\left.{\atop 0,-c, -\alpha}\; \right | \zeta \right)+\mathcal O(n^{-2}),\\
n^{-\alpha -1}\le(\frac {n}{n+1}\ri)^{\frac {\alpha(\alpha+1)}2} G^{2,1}_{2,3} \left(\left.{-\alpha -n,n\atop 0,-c, -\alpha}\; \right |
x
\right)
=G^{2,0}_{0,3} \left(\left.{\atop 0,-c, -\alpha}\; \right | \zeta \right)+\mathcal O(n^{-2}).  
\end{align*}

Consequently, the asymptotic behaviour of the four kernels is

\bea
 &\& n^{-2\alpha -2}\le(\frac {n}{n+1}\ri)^{\alpha(\alpha+1)}\hspace{-10pt}
 K^{(n)}_{00}\le (
 x,y
 \ri) \to \int_0^1 G_{a}(t\zeta ) G_{b}(t\xi ) t^\alpha \d t,
 \\
 &\&
 n^{-2a -2}\le(\frac {n}{n+1}\ri)^{\alpha a+\alpha} K^{(n)}_{01}
\le (
x,y
\ri) 
\to \xi^b \int_0^1 G_{a}(t\zeta) \wt G_{b}(t\xi) t^\alpha \d t, \\
 &\&
 n^{-2b-2} \le(\frac {n}{n+1}\ri)^{\alpha b+\alpha} 
 K^{(n)}_{10}
 \le (
 x,y
 \ri) 
 \to \zeta^a \int_0^1 \wt G_{a}(t\zeta) G_{b}(t\xi) t^\alpha \d t, \\
 &\&
 n^{-2} \le(\frac {n}{n+1}\ri)^{\alpha}K^{(n)}_{11}
 \le (
 x,y
 \ri) 
 \to  \zeta^a \xi ^b\int_0^1 \wt G_{a}(t\zeta ) \wt G_{b}(t\xi) t^\alpha \d t  - \frac{1}{\zeta +\xi }, 
\eea
and the convergence of the above limits is all within $\mathcal O(n^{-2})$.  
\et
\begin{proof} 
The proof of this theorem is immediate from the Stirling approximation formula in Lemma \ref{lemmagamma}, the explicit integral expressions for $G_{c,n}, \wt G_{c,n}$ in \eqref{G}, \eqref{wtG} and the convolution form of the kernels in Theorems \ref{K00thm}, \ref{K10thm}, \ref{K11thm}. \end{proof}

Now we turn our attention to the $H$ kernels defined in \eqref{eq:Hs}.  
Using the rescaling \eqref{rescaling} and the asymptotics of the kernels in Theorem \ref{kerasymptthm} we find the following expressions
\bea
\lim_{n\to \infty}n^{-2}  \le(\frac {n}{n+1}\ri)^{\alpha}H^{(n)}_{00} (x,y) &\&=  \zeta^a \xi ^b  \int_0^1 G_{a}(t\zeta) G_{b}(t\xi) t^\alpha \d t, 
\\
\lim_{n\to \infty}n^{-2} \le(\frac {n}{n+1}\ri)^{\alpha} H^{(n)}_{01}(x,y) &\& =\zeta^{ a}\xi^{b}   \int_0^1 G_{a}(t\zeta) \wt G_{b}(t\xi) t^\alpha \d t, \\
\lim_{n\to \infty}n^{-2}  \le(\frac {n}{n+1}\ri)^{\alpha}H^{(n)}_{10}(x,y) &\& = \zeta ^{a} \xi^{b } \int_0^1 \wt G_{a}(t\zeta) G_{b}(t\xi) t^\alpha \d t, \\
\lim_{n\to \infty}n^{-2}  \le(\frac {n}{n+1}\ri)^{\alpha}H^{(n)}_{11}(x,y) &\& =  \zeta^{a}\xi^{b}
\int_0^1 \wt G_{a}(t\zeta) \wt G_{b}(t\xi ) t^\alpha \d t  - \frac{1}{\zeta+\xi}. 
\eea
\br
The scaling factor $n^{-2}$ in front of the kernels $H^{(n)}_{\mu\nu}$ is natural since kernels transform as $\sqrt{\d x \d y}$. The convergence to the limits is at the rate $\mathcal O(n^{-2})$ (as per Theorem \ref{kerasymptthm}).
\er
The limiting kernels above define a novel class of two-level  random point fields.  
\bp[Meijer-G random point field and universal class]
\label{MeijerGK}
In the scaling limit the correlations of the eigenvalues of $M_1,M_2$ are determined by the two--level random point field with kernels
\begin{subequations}
\begin{align}\label{eq:Gs0}
\mathcal G_{00} (\zeta,\xi) &=   \zeta^a \xi^b \int_0^1 G_{a}(t\zeta) G_{b}(t\xi) t^\alpha \d t, &
\mathcal G_{01}(\zeta,\xi) &=\zeta^a \xi ^{b}   \int_0^1 G_{a}(t\zeta) \wt G_{b}(t\xi) t^\alpha \d t, \\
\mathcal G_{10}(\zeta,\xi) &= \zeta^{a}\xi ^b \int_0^1 \wt G_{a}(t\zeta) G_{b}(t\xi) t^\alpha \d t,&
\mathcal G_{11}(\zeta,\xi) &= \zeta^{a}\xi^{b} 
\int_0^1 \wt G_{a}(t\zeta) \wt G_{b}(t\xi ) t^\alpha \d t  - \frac{1}{\zeta+\xi},  
\end{align}
\end{subequations}
where $\xi,\zeta \in\R_+$.  
\ep

If we absorb the powers in the definition of $H_a,\wt H_a$ as in Definition \ref{defHfun} then  the formulas become even more symmetric.  
\bc \label{MeijerGH}
In the scaling limit the correlations of the eigenvalues of $M_1,M_2$ are determined by the two--level random point field with kernels
\begin{subequations}
\begin{align}\label{eq:Gs}
\mathcal G_{00} (\zeta,\xi) &=   \int_0^1 H_{a}(t\zeta) H_{b}(t\xi ) \d t, &
\mathcal G_{01}(\zeta,\xi) &=\int_0^1 H_{a}(t\zeta) \wt H_{b}(t\xi) \d t, \\
\mathcal G_{10}(\zeta,\xi) &= \int_0^1 \wt H_{a}(t\zeta) H_{b}(t\xi) \d t,&
\mathcal G_{11}(\zeta,\xi) &= 
\int_0^1 \wt H_{a}(t\zeta) \wt H_{b}(t\xi ) \d t  - \frac{1}{\zeta+\xi},  
\end{align}
\end{subequations}
where $\xi,\zeta \in\R_+$ and the kernels are given by the Meijer $G$ functions 
\begin{equation*}
H_c(z)=G^{1,0}_{0,3} \left(\left.{\atop c,0, -\alpha+c}\; \right | z \right), 
\quad \wt H_c(z)=G^{2,0}_{0,3} \left(\left.{\atop c,0, -\alpha+c}\; \right | z \right).  
\end{equation*}

\ec

Thus we have completed the proof of Theorem \ref{MeijerGKprime} stated in the introduction.

\br  It is easy to see by direct computation that the functions $H_c, \tilde H_c$ are connected to 
hypergeometric functions $~_0F_2$.  For example 
\begin{equation} \label{eq:Ha}
H_a(z)={\frac {{z}^{a}~_0F_2\left( {\atop 1+a,1+\alpha},-z \right) }{
\Gamma  \left( 1+a \right) \Gamma  \left( 1+\alpha \right) }}, 
\end{equation}
\er
while $\tilde H_a(z)$ is, for $a\neq 0,1,\cdots$, a linear combination of two 
hypergeometric functions of type $~_0F_2$
\begin{equation} \label{eq:tilHa} 
\tilde H_a(z)={\frac {\Gamma  \left( a \right) ~_0F_2\left( {\atop 1+b,1-a},z
 \right) }{\Gamma  \left( 1+b \right) }}+{\frac {\Gamma  \left( -a
 \right) {z}^{a}~_0F_2\left( {\atop 1+a,1+\alpha},z \right) }{
\Gamma  \left( 1+\alpha \right) }}. 
\end{equation}
 The remaining cases $a=n=0,1,\cdots$ are handled by 
 taking limits $a\rightarrow n$.  For example
 \begin{equation*}
\wt H_0(z) =-\frac{\ln z}{\Gamma(1 + b)} + \mathcal O(1)\ \ \ \ \hbox { if } a=0.  
\end{equation*}

\subsubsection{Expressions of the kernels in ``integrable'' form}
We note that the functions $H_a(t), \wt H_a(t)$ solve differential equations of the third order; this fact is implied by their identification as Meijer-G functions.  A simple derivation is as follows: denote by $A(u)$ the rational expression  of $\Gamma$ functions in the integrand of $H_c$ or $\wt H_c$; then a direct inspection shows 
\bea
A(u+1) = \mp u (u+c)(\alpha - c -u) A(u)  \ \ \ \Rightarrow \ \ \  \ 
\Delta_\zeta (\Delta_\zeta + \alpha - c) (\Delta_\zeta - c) f(\zeta) = \mp \zeta f(\zeta), 
\eea
where $\Delta_\zeta=\zeta \frac{d}{d\zeta}$, the upper sign is for $H_c$, the lower for $\wt H_c$ respectively.  
Now, let $c\in \{a, b\}$ and 
denote by $f_t(\zeta):= f (t\zeta)$, $g_t(\xi):= g(t\xi)$, and let $f$ be one of $H_{a}$ or $\wt H_a$ and $g$ one of $H_b$ or $\wt H_b$. Since $\Delta $ is scale invariant we have 
\begin{subequations}
\begin{align} 
 &g_t(\xi)\Delta_\zeta (\Delta_\zeta + b) (\Delta_\zeta - a) f_{t}(\zeta) =\mp t\zeta f_t (\zeta) g_{t}(\xi), \\
 & f_{t}(\zeta) \Delta_\xi (\Delta_\xi + a) (\Delta_\xi - b) g_{t}(\xi) = \mp t\xi f_{t}  (\zeta) g_{t}(\xi).  
\end{align}
\end{subequations}
Adding these two equations together, dividing by $t$ and using that $\Delta_\zeta f_t(\zeta) = \Delta_t f_t(\zeta)$ (and similarly for $g_t(\xi)$) we find 
\begin{equation}
\pa_t \big[g_t \Delta^2_t f_t-\Delta_t g_t \Delta_t f_t+f_t \Delta^2_t g_t +(b-a)(g_t \Delta_t f-
f_t \Delta_t g_t)-ab f_t g_t\big ]=
(\mp \zeta +\mp \xi) f_t g_t.  
\end{equation} 
Integrating with respect to $t$ from $0$ to $0<\tau$ we obtain 
\begin{multline}
(\mp\zeta+\mp \xi)\int_{0}^\tau  f_t g_t \d t =   f_t\Delta_t^2 g_t - \Delta_tf_t \Delta_t g_t +  g_t \Delta_t^2 f_t 
  +(b-a)\le( g_t\Delta_t f_t - f_t\Delta_t g \ri) 
 - ab \,f_t g_t\Bigg|_{0+}^\tau=\\
\le[f(t\zeta), \Delta_\zeta f(t\zeta), \Delta^2_\zeta f(t\zeta)\ri] \le[
 \begin{array}{ccc}
 -ab & a-b &1\\
 b-a& -1 & 0\\
 1&0&0
 \end{array}
 \ri]\le[\begin{array}{c}
 g(t\xi)\\
 \Delta_\xi g(t\xi)\\
 \Delta^2_\xi  g(t\xi)
 \end{array}\ri]_{t=0+}^\tau, 
 \label{eval}
\end{multline}
where in the last line we used again the scale invariance of $\Delta$.  
The last computation motivates the following definition
\bd
\label{defB}
Let $f=f(\zeta)$ and $g=g(\xi)$.  Then the  {\bf point-split bilinear concomitant} is defined as 
\begin{multline}
\mathcal B(f, g):=  f\Delta_\xi^2 g - (\Delta_\zeta f) (\Delta_\xi g) +  g \Delta^2_\zeta f  
  +(b-a)\le( g\Delta_\zeta f - f\Delta_\xi g \ri) 
 - ab \,f g
 =\\
 \le[f, \Delta_\zeta f, \Delta^2_\zeta f\ri] \le[
 \begin{array}{ccc}
 -ab & a-b &1\\
 b-a& -1 & 0\\
 1&0&0
 \end{array}
 \ri]\le[\begin{array}{c}
 g\\
 \Delta_\xi g \\
 \Delta^2_\xi  g
 \end{array}\ri].  \label{bilconc} 
\end{multline}

\ed
 If $\zeta=\pm \xi$, suitably chosen to make the left side of \eqref{eval} vanish, then \eqref{eval} implies $\pa_\xi \mathcal B(f,g)=0$. This is a form of the {\em bilinear concomitant} \cite{InceBook} for equations in duality, which explains our naming convention.
 
From this point onward we will be interested in $\tau=1$.  
The evaluation of the right hand side of \eqref{eval} at $t=0+$  is case dependent. 

\paragraph{Case $H_a, H_b$.}
If $f=H_a,\ g=H_b$ then the evaluation at  $t=0$ vanishes:
indeed by \eqref{eq:Ha}
\be
H_c(\zeta) = \zeta^c\le(\frac 1{\Gamma( 1+c)\Gamma( 1+\alpha ) } + \mathcal O(\zeta)\ri), 
\ee
and it suffices therefore to verify $\mathcal B(\zeta^a,\xi^b)=0$ and
$$ \lim_{t\rightarrow 0+} \mathcal B((t\zeta)^{a+k},(t\xi)^{b+l})=\lim_{t\rightarrow 0+} t^{\alpha +k+l} \mathcal B(\zeta^{a+k},\xi^{b+l}) =0, \text{ for } 1\leq k+l\ ,
$$
 since $\alpha+1>0$.

\paragraph{Case $\wt H_a, H_b$ or vice versa.}
This case is similar to the previous case; one uses \eqref{eq:tilHa} and analyzes cases.  The relevant points to remember are that  by our assumptions $a+1>0, b+1>0$ and $\alpha +1>0$.  
In all cases $\mathcal B$ vanishes at $t=0+$.  

\paragraph{Case $\wt H_a, \wt H_b$.}
In this case the $\lim_{t\rightarrow 0+} \mathcal B(\wt H_a,\wt H_b)(t\zeta,t\xi)=-1$.  This follows from the 
fact that in the expansion of $\wt H_a$ (or $\wt H_b$) there is one more pairing 
than in previous cases.  This is the pairing of the leading term in \eqref{eq:tilHa}, namely in 
${\frac {\Gamma  \left( a \right) ~_0F_2\left( {\atop 1+b,1-a},z
 \right) }{\Gamma  \left( 1+b \right) }}$, with its analog in $\wt H_b$. Extracting the leading term 
 leads to $\mathcal B(\frac{\Gamma(a)}{\Gamma(1+b)}, \frac{\Gamma(b)}{\Gamma(1+a)})=-ab\frac{\Gamma(a)\Gamma(b)}{\Gamma(1+a)\Gamma(1+b)}=-1$.    

It is now elementary to substitute the above results into \eqref{eval} and, further,  
into the definition of $\mathcal G$ kernels given by \eqref{eq:Gs}.  
\bp
\label{propConcomitant}
Let $\mathcal B$ be the point-split bilinear concomitant defined by \eqref{defB}.  
Then the correlation kernels of the Meijer-G random point field satisfy 
\begin{subequations}
\begin{align}
&\mathcal G_{00}(\zeta,\xi)=- \frac{\mathcal B(H_a(\zeta),H_b(\xi))}{\zeta+\xi},& 
 & \mathcal G_{01}(\zeta,\xi)=\frac{\mathcal B(H_a(\zeta),\wt H_b(\xi))}{-\zeta+\xi}, \\
 &\mathcal G_{10}(\zeta,\xi)=\frac{\mathcal B(\wt H_a(\zeta),H_b(\xi))}{\zeta-\xi},&
 & \mathcal G_{11}(\zeta,\xi)=\frac{\mathcal B(\wt H_a(\zeta),\wt H_b(\xi))}{\zeta+\xi}.  
\end{align}
\end{subequations}
\ep

\paragraph{Acknowledgements}
This paper was initiated at the Banff International Research Station. We thank BIRS for the hospitality and for providing excellent work conditions. M.B. wishes to thank SISSA (Trieste) for hospitality during which part of the work was completed. 
 M. B. and J.S. acknowledge a support by Natural Sciences and Engineering Research Council of Canada. M.G. is partially supported by the NSF grant DMS-1101462.

\appendix
\renewcommand{\theequation}{\Alph{section}.\arabic{equation}}

\section{The Meijer-G functions}
\label{MGapp}
For the convenience of the reader we recall the definition of the Meijer-G function
\be
G_{p,q}^{\,m,n} \!\left( \le.{ a_1, \dots, a_p \atop  b_1, \dots, b_q } \; \right| \, z \right) = \int_\gamma \frac {\d u }{2\pi i}\, \frac{\prod_{j=1}^m \Gamma(b_j + u) \prod_{j=1}^n \Gamma(1 - a_j -u)} {\prod_{j=m+1}^q \Gamma(1 - b_j -u ) \prod_{j=n+1}^p \Gamma(a_j +u )} \,z^{-u}.  
\ee
The contour $\gamma$ is a contour that leaves the poles of the $\Gamma(b_j + u)$'s to the left and the poles of the $\Gamma( 1 - a_j - u)$'s to the right (the implicit assumption is that none of the poles of the former coincides with any of the poles of the latter). The contour extends to infinity with $|\arg(u)| <\frac \pi 2$ (in the right half plane) if $p>q$, otherwise it extends to  negative infinity with $\frac \pi 2<\arg(u) <\frac { 3 \pi }2$ (in the left half plane).  It is the latter case that is relevant for us.  
In the paper, we assume that the branch of $z^{-u}$ is chosen in such a way that the cut is 
for $z<0$.  
There is an considerable applied mathematics literature on $G$-functions (e.g. \cite{Luke}).  
For the reader's convenience we collect several formulas used in the paper.  
\begin{enumerate} 
\item Shifting formula: \begin{equation}\label{eq:shifting} 
z^{\sigma} G_{\, p,q}^{\,m,n} \!\left({\mathbf{a_p} \atop  \mathbf{b_q }} \; | \, z \right) =
G_{\, p,q}^{\,m,n} \!\left(\left.{\mathbf{a_p} +\sigma\atop  \mathbf{b_q }+\sigma} \; \right | \, z \right), 
\end{equation} 
where $+ \sigma$ means shifting every entry by $\sigma$.  
\item Integrals containing products of two $G$-functions: 
\begin{equation}
\int_0^{\infty} G_{p,q}^{\,m,n} \!\left( \left. \begin{matrix} \mathbf{a_p} \\ \mathbf{b_q} \end{matrix} \; \right| \, \frac{x}{\eta} \right)
G_{\sigma, \tau}^{\,\mu, \nu} \!\left( \left. \begin{matrix} \mathbf{c_{\sigma}} \\ \mathbf{d_\tau} \end{matrix} \; \right| \, \omega x \right) dx 
= \eta  \; G_{q + \sigma ,\, p + \tau}^{\,n + \mu ,\, m + \nu} \!\left( \left. \begin{matrix} - b_1, \dots, - b_m, \mathbf{c_{\sigma}}, - b_{m+1}, \dots, - b_q \\ - a_1, \dots, -a_n, \mathbf{d_\tau} , - a_{n+1}, \dots, - a_p \end{matrix} \; \right| \, \eta \omega \right).  
\end{equation}
Since $G_{1,0}^{\,1,1} \!\left( \left. \begin{matrix} a \\ b \end{matrix} \; \right| \, -z\right)$
is proportional to $z^b ~_1F_0(1+b-a,- z)=z^b\frac{1}{(1+z)^{1+b-a}}$ we obtain \cite{Luke}
\begin{equation}\label{eq:Gintegral} 
\int_0^{\infty} \frac{y^{\alpha -1}}{(z+y)^{\sigma} }G_{\, p,q}^{\,m,n} \!\left(\left.{\mathbf{a_p} \atop  \mathbf{b_q }} \; \right | \, \mu y \right) \d y =\frac{z^{\alpha-\sigma}}{\Gamma(\sigma)}G_{\, p+1,q+1}^{\,m+1,n+1} \!\left(\left.{1-\alpha,\mathbf{a_p} \atop  \sigma-\alpha, \mathbf{b_q }} \; \right | \, \mu z\right).  
\end{equation} 
\item For $p\leq q$ 
\begin{equation} \label{eq:Gzero} 
G_{\, p,q}^{\,0,n} \!\left({\mathbf{a_p} \atop  \mathbf{b_q }} \; | \, z \right) =0.  
\end{equation} 
\end{enumerate}

\section{Correlation functions}\label{Correlapp}
In this appendix we recall the definitions and formulas needed to study the correlation functions for generic Cauchy matrix models (no restrictions on the measures).  We use the notation that is slightly different than the one used in \cite{Bertola:CauchyMM} to accommodate the needs of the present paper.  
Given two measures $d\alpha(x) =\alpha(x)\d x, \, \d \beta(y)=\beta(y)\d y$ with densities $\alpha(x),\beta(y)$ respectively, both supported on the positive half-line $\R_+$, the Cauchy matrix 
model is the probability measure on the pairs $(M_1,M_2)$ of positive semi-definite 
$N\times N$ matrices given by
$\ds 
\d \mu(M_1,M_2)=\frac{ \alpha(M_1) \beta(M_2) \d M_1 \d M_2}{\mathcal Z_N \det(M_1+M_2)^N}, $
where $\alpha(M)$ (or $\beta(B)$) stands for the induced measure on the spectrum of $M$, i.e. $\prod_j \alpha(x_j)$ where $x_j$ are the (positive) eigenvalues of $M$.
Let us introduce two families of polynomials $\{p_j(x), j=0,1,.. \}$ and $\{q_j(y), j=0,1,..\}$ 
which are biorthonormal with respect to the pairing with the Cauchy kernel
\begin{equation*}
\iint _{\R _+^2} \frac{p_j(x) \alpha(x) q_k(y) \beta(y)} {x+y} \d x \d y =\delta_{j,k}, 
\end{equation*}
with the technical proviso that the leading coefficients are identical and positive 
to render all polynomials unique.  Following \cite{Bertola:CauchyMM}, but changing slightly the notation, we introduce four kernels used in the present paper
\begin{align*}
& K_{00}^{(n)}(x,y):= \sum_{j=0}^{n-1} p_j(x)q_j(y), & &K_{01}^{(n)} (x,y):= \int K_{00}^{(n)} (x',y) \frac{\alpha(x') \d x'}{x+x'},\\
&K_{10}^{(n)} (x,y):= \int K_ {00}^{(n)}(x,y')\frac {\beta(y') \d y'}{y+y'},& 
&K_{11}^{(n)} (x,y):= \iint K_{00}^{(n)} (x',y')\frac{\alpha(x') \d x'\beta(y') \d y'}{(x+x')(y+y')} - \frac1{x+y}\ .\label{Kkernels}
\end{align*}
With this notation  in place the correlation functions for $r$ eigenvalues $x_1,\dots, x_r$ of $M_1$ and $s$ eigenvalues $y_1,\dots, y_s$ of $M_2$ can be shown to be given by 
\be
\rho_{(r,s)}(x_1,\dots, x_r; y_1,\dots, y_s) = \prod_{j=1}^r \alpha(x_j)\ \prod_{k=1}^s \beta(y_k) 
\det\le[
\begin{array}{c|c}
\big[K_{01}^{(n)} (x_i,x_j)\big]_{i,j\leq r} & \big[K_{00}^{(n)}  (x_i, y_j) \big]_{i\leq r, j\leq s}\\[10pt]
\hline
\rule{0pt}{16pt}\big[K_{11}^{(n)} (y_i,x_j)\big]_{i\leq s, j\leq r} & \big[K_{10}^{(n)}  (y_i, y_j) \big]_{i,j \leq s}  
\end{array}
\ri]
\ee
It is easy to write the correlation function as one determinant by including the 
products in front of the determinant above as determinants of diagonal matrices.  
This is not a unique procedure but the choice that works naturally for the present paper is the following.  We write symbolically 
\begin{equation*}
\rho_{(r,s)}(x_1,\dots, x_r; y_1,\dots, y_s)=\det \le( \begin{bmatrix} 
\mathbf {\alpha(x)}& 0\\0& \mathbf I \end{bmatrix} \begin{bmatrix} \mathbf {K_{01}(x,x)}& \mathbf {K_{00}(x,y)} \\ \mathbf{ K_{11}(y,x)} & \mathbf {K_{10}(y,y)} \end{bmatrix} \begin{bmatrix} \mathbf I&0\\0& \mathbf{\beta(y)} \end{bmatrix} \ri), 
\end{equation*}
where $\mathbf{\alpha(x)}=\text{diag} (\alpha(x_1),\cdots,\alpha(x_r)), \text{diag} (\beta(y_1),\cdots,\beta(y_s))$ respectively.  Hence 
\begin{equation*}
\rho_{(r,s)}(x_1,\dots, x_r; y_1,\dots, y_s)=\det \le( \begin{bmatrix} 
\mathbf {\alpha(x)}\mathbf {K_{01}(x,x)}& \mathbf{\alpha(x)} \mathbf {K_{00}(x,y)}\mathbf{\beta(y)} \\ \mathbf{ K_{11}(y,x)} & \mathbf {K_{10}(y,y)\mathbf{\beta(y)}} \end{bmatrix}  \ri).  
\end{equation*}
This leads to the definition of new kernels: 
\begin{subequations}
\begin{align}
H_{00}^{(n)}(x,y)&:=\alpha(x) \beta(y) K_{00}^{(n)}(x,y),& 
H_{01}^{(n)}(x,y)&:=\alpha(x)K_{01}^{(n)} (x,y),\\
H_{10}^{(n)}(x,y)&:=\beta(y)K_{10}^{(n)} (x,y),&
H_{11}^{(n)}(x,y)&:=K_{11}^{(n)} (x,y) .\label{eq:Hkernels}
\end{align}
\end{subequations}
\section{From the Meijer-G to the  Bessel field}
\label{MtoB}

Starting from the formula for $\mathcal G_{0,1}(\zeta,\xi)$  in \eqref{g00g01} one has for $b\to \infty$ 
\bea
&\& \frac {b}4\,\mathcal G_{01}\le(\frac {b  \zeta}4 , \frac{b\xi}4 \ri) =\frac b4  \int_0^1 \d t \int_{\gamma^2} \frac {\d u \d v}{(2i\pi)^2} \frac {\Gamma(u+a) \Gamma(v)\Gamma(v+b)}{\Gamma(1-u)\Gamma(a+1-v)\Gamma(b+1-u)} \le(\frac {t\zeta b}{4} \ri)^{-u}\le(\frac {t\xi b}{4} \ri)^{-v} \mathop{=}^{\tiny Lemma \ref{lemmagamma}} \\
&\&= \frac b 4 \int_0^1 \d t \int_{\gamma^2} \frac {\d u \d v}{(2i\pi)^2} \frac {\Gamma(u+a) \Gamma(v) b^{v+u-1}}{\Gamma(1-u)\Gamma(a+1-v)} \le(\frac {t\zeta b}{4} \ri)^{-u}\le( \frac {t\xi b}{4} \ri)^{-v}(1 + \mathcal O(b^{-1}) )=\\
&\& =\frac 14\int_0^1 \d t \int_{\gamma^2} \frac {\d u \d v}{(2i\pi)^2} \frac {\Gamma(u+a) \Gamma(v)}{\Gamma(1-u)\Gamma(a+1-v)} \le(\frac {t\zeta}{4} \ri)^{-u}\le( \frac {t\xi}{4} \ri)^{-v}(1 + \mathcal O(b^{-1}) )
\mathop{=}^{\eqref{BesselKernel}}\\
&\&=  \le(\frac \zeta \xi\ri)^\frac a 2 K_{_B,a}(\zeta,\xi)\,(1 + \mathcal O (b^{-1}))
\eea
where we have shifted the variable $u\mapsto u-a$ to compare with the expressions in \eqref{BesselKernel}.
To see that the remaining kernels vanish let us perform similar computations for $\mathcal G_{10} \text{ and } \mathcal G_{00}$.  We obtain 
\bea
&\& \frac {b}4\,\mathcal G_{00}\le(\frac {b  \zeta}4 , \frac{b\xi}4 \ri) =\frac b4  \int_0^1 \d t \int_{\gamma^2} \frac {\d u \d v}{(2i\pi)^2} \frac {\Gamma(u+a) \Gamma(v+b)}{\Gamma(1-u)\Gamma(1-v) \Gamma(a+1-v)\Gamma(b+1-u)} \le(\frac {t\zeta b}{4} \ri)^{-u}\le(\frac {t\xi b}{4} \ri)^{-v} \mathop{=}^{\tiny Lemma \ref{lemmagamma}} \\
&\& =\frac 14\int_0^1 \d t \int_{\gamma^2} \frac {\d u \d v}{(2i\pi)^2} \frac {\Gamma(u+a)}{\Gamma(1-u)\Gamma(1-v)\Gamma(a+1-v)} \le(\frac {t\zeta}{4} \ri)^{-u}\le( \frac {t\xi}{4} \ri)^{-v}(1 + \mathcal O(b^{-1}) )=\mathcal O(b^{-1}).  
\eea
The reason  the integral vanishes (to a leading order), is because the integration in $v$ has no singularities within the contour $\gamma$, which can be retracted to $-\infty$ and easily estimated to give zero. Similarly 
\bea
&\& \frac {b}4\,\mathcal G_{10}\le(\frac {b  \zeta}4 , \frac{b\xi}4 \ri) =\frac b4  \int_0^1 \d t \int_{\gamma^2} \frac {\d u \d v}{(2i\pi)^2} \frac {\Gamma(u+b) \Gamma(v)\Gamma(v+a)}{\Gamma(1-u)\Gamma(b+1-v)\Gamma(a+1-u)} \le(\frac {t\zeta b}{4} \ri)^{-u}\le(\frac {t\xi b}{4} \ri)^{-v} \mathop{=}^{\tiny Lemma \ref{lemmagamma}} \\
&\& =\frac 14\int_0^1 \d t \int_{\gamma^2} \frac {\d u \d v}{(2i\pi)^2} \frac {\Gamma(v+a) \Gamma(v)}{\Gamma(1-u)\Gamma(a+1-u)} \le(\frac {t\zeta}{4} \ri)^{-u}\le( \frac {t\xi}{4} \ri)^{-v}(1 + \mathcal O(b^{-1}) ) = \mathcal O(b^{-1}), 
\eea
where this time the leading term vanishes because of the $u$-integration.
Finally 
\bea
&\& \frac {b}4\,\mathcal G_{11}\le(\frac {b  \zeta}4 , \frac{b\xi}4 \ri) =
\frac b4  \int_0^1 \d t \int_{\gamma^2} \frac {\d u \d v}{(2i\pi)^2} 
\frac {\Gamma(u)\Gamma(v)\Gamma(u+a)\Gamma(v+b)}{\Gamma(b+1-u)\Gamma(a+1-v)}
 \le(\frac {t\zeta b}{4} \ri)^{-u}\le(\frac {t\xi b}{4} \ri)^{-v} - \frac 1{\zeta +\xi} \mathop{=}^{\tiny Lemma \ref{lemmagamma}} \\
&\& =\frac 14\int_0^1 \d t \int_{\gamma^2} \frac {\d u \d v}{(2i\pi)^2}
\frac {\Gamma(u)\Gamma(v)\Gamma(u+a)}{\Gamma(a+1-v)}
  \le(\frac {t\zeta}{4} \ri)^{-u}\le( \frac {t\xi}{4} \ri)^{-v}(1 + \mathcal O(b^{-1}) )- \frac 1{\zeta + \xi}
\eea
As the reader can see, the kernel has a limit that does not identically vanish; however this is inconsequential for the correlation functions, since this kernel appears in the lower-left block of the determinant \eqref{1-12} and since the opposite block containing $\mathcal G_{00}$ tends to zero, the correlation function will not contain $\mathcal G_{11}$ either (in the leading order). Thus all correlation functions for the eigenvalues of $M_1$ (in the scaling limit) will behave like the Bessel DRPF, while those involving the $-$ field tend to zero uniformly over compact sets.
The computations for the kernels $H^{(n)}_{\mu\nu}$ with the scaling \eqref{219} are similar and we omit them.

%

\def\cprime{$'$}

%
%
%
%
%
%
%
%

\end{document}

%% file: MeijerContour.pdf_t
\begin{picture}(0,0)%
\includegraphics{MeijerContour.pdf}%
\end{picture}%
\setlength{\unitlength}{3947sp}%
\begingroup\makeatletter\ifx\SetFigFont\undefined%
\gdef\SetFigFont#1#2#3#4#5{%
  \reset@font\fontsize{#1}{#2pt}%
  \fontfamily{#3}\fontseries{#4}\fontshape{#5}%
  \selectfont}%
\fi\endgroup%
\begin{picture}(5799,3624)(2164,-5173)
\end{picture}%